\documentclass[3p]{elsarticle}

\usepackage{amsthm}
\usepackage{amsmath} 
\usepackage{amssymb}

\usepackage[table,dvipsnames]{xcolor}

\usepackage{enumerate}
\usepackage{hyperref}

\newtheorem{Ass}{Assumption} 

\newtheorem{theorem}{Theorem}[section]
\newtheorem{lemma}[theorem]{Lemma}
\newtheorem{proposition}[theorem]{Proposition}
\newtheorem{corollary}[theorem]{Corollary}

\newtheorem{definition}[theorem]{Definition}
\newtheorem{example}[theorem]{Example}

\newtheorem{remark}[theorem]{Remark}

\newcounter{syscounter}
\newenvironment{sysnum}{\begin{list}{($\Sigma{\arabic{syscounter}}$)}%
{\settowidth{\labelwidth}{($\Sigma4$)}
\settowidth{\leftmargin}{($\Sigma4$)~}%
\usecounter{syscounter}}}
{\end{list}}

\newcommand \N   {\mathbb{N}}
\newcommand \R   {\mathbb{R}}
\newcommand \C   {\mathbb{C}}

\newcommand \K   {\mathcal{K}}
\newcommand \Kinf{\mathcal{K_\infty}}

\newcommand \KL  {\mathcal{KL}}
\newcommand \LL  {\mathcal{L}}

\newcommand \Dc   {\mathcal{D}}

\newcommand \qrq   {\quad\Rightarrow\quad}

\newcommand \Iff   {\Leftrightarrow}

\newcommand \eps {\varepsilon}

\newcommand{\abs}[1]{\left|#1\right|}
\newcommand{\norm}[1]{\left\Vert #1\right\Vert}
\newcommand \Sigmafp {0}

\journal{}

\begin{document}

\begin{frontmatter}

\title{Non-coercive Lyapunov functions for infinite-dimensional systems}

\author[Pas]{Andrii Mironchenko}
\ead{andrii.mironchenko@uni-passau.de}

\author[Pas]{Fabian Wirth}
\ead{fabian.(lastname)@uni-passau.de}

\address[Pas]{Faculty of Computer Science and Mathematics, University of Passau,
Innstra\ss e 33, 94032 Passau, Germany }

\begin{abstract}
We show that the existence of a non-coercive Lyapunov function is sufficient for uniform
global asymptotic stability (UGAS) of infinite-dimensional systems with
external disturbances provided the speed of decay is measured in terms of
the norm of the state and an additional mild assumption is satisfied. 
For evolution equations in Banach spaces with Lipschitz continuous
nonlinearities these additional assumptions become especially simple. The
results encompass some recent results on linear switched systems on Banach
spaces. 
Finally, we derive new non-coercive converse Lyapunov theorems and give
some examples showing the necessity of our assumptions.
\end{abstract}



\end{frontmatter}

\section{Introduction}
\label{sec:introduction}

The theory of Lyapunov functions is one of the cornerstones in the
analysis and synthesis of dynamical systems. Since its origins due to
Lyapunov there have been numerous developments leading first to sufficient 
and later on to necessary conditions for various dynamical properties
expressed in terms of Lyapunov functions, see \cite{Yos66,Hah67,LSW96,Kel15}. The uses of Lyapunov functions are manifold. Originally invented
by Lyapunov to characterize stability properties of fixed points, or more
complex attractors, they have become useful in other contexts. In existence
theory they provide bounds on the transients of solutions and so in some
situations conditions for forward completeness of trajectories,
\cite{AnS99}. A criterion for the existence of a bounded absorbing ball is
formulated in \cite[Theorem 2.1.2]{chueshov2015dynamics}.
In the finite-dimensional
case they have also been used to investigate the geometric structure of
the general solution of differential equations \cite{wilson1967structure}
and to analyze coordinate-free notions of growth rates
\cite{grune1999asymptotic}. The complete characterization of
attractor-repeller pairs of continuous-time dynamical systems has been
obtained in \cite{wilson1973lyapunov,conley1988gradient}. Some of these uses extend from
finite-dimensional applications to the infinite-dimensional case, while
others use distinct finite-dimensional arguments.

On the other hand numerous converse results have been obtained which prove
the existence of certain types of Lyapunov functions characterizing
different stability notions, \cite{Kel15}.
Before starting to look for a Lyapunov function it is highly desirable to
know in advance that such a Lyapunov function for a given class of systems
exists.  The first results guaranteeing existence of Lyapunov functions
for asymptotically stable systems appeared in the works of Kurzweil
\cite{Kur56} and Massera \cite{Mas56}.  These have been
generalized in different directions
\cite{Hen81,Dei92,KaJ11b,CLS98,Kel15,Sch15, TeP00}.  In particular, in
\cite{KaJ11b, TeP00} the well-known Sontag's $\KL$-lemma \cite{Son98} has
been elegantly used in order to prove global Lyapunov theorems with the
help of Yoshizawa's method \cite[Theorem 19.3]{Yos66},
\cite[Theorem 4.2.1]{Hen81}.  

In this paper we address a particular aspect of infinite-dimensional
systems by providing Lyapunov stability results for non-coercive Lyapunov
functions and converse theorems. Throughout the paper we will be dealing
with a fixed point at the origin, $x^*=0$, which imposes no restrictions
as long as we are interested in the stability properties of fixed
points. The question of coercivity of Lyapunov functions is subtle and
statements that are obvious in finite dimensions become intricate or even
wrong in infinite dimensions. The standard definition of a Lyapunov
function $V$, found in many textbooks on finite-dimensional dynamical
systems, is that it should be a continuous (or more regular) positive
definite and proper function, i.e. a function for which there exist
$\Kinf$\footnote{Increasing, unbounded, continuous, positive definite
  function.} functions $\psi_1,\psi_2,\alpha$ such that
\begin{equation}
    \label{eq:2}
    \psi_1(\|x\|) \leq V(x) \leq \psi_2(\|x\|) \quad \forall x \in X,
\end{equation}
and such that 
\begin{equation}
    \label{eq:3}
    \dot V(x) < - \alpha(\|x\|) \quad \forall x \in X\backslash\{0\}, 
\end{equation}
where $\dot V(x)$ is some sort of generalized derivative of $V$ along the
trajectories of the system, see below for a precise definition. Of course, if we have \eqref{eq:2} then we
may just as well require that there is a $\gamma\in\Kinf$ such that 
\begin{equation}
    \label{eq:3b}
    \dot V(x) < - \gamma(V(x)),
\end{equation}
because in the presence of \eqref{eq:2} we clearly have an equivalence of
\eqref{eq:3} and \eqref{eq:3b}. 

The inequality \eqref{eq:3b} shows that $V(x(t))$ converges to zero in a
uniform way as $t \to \infty$ (by the "comparison principle"), and \eqref{eq:2} implies that $\|x(t)\|$
has the same asymptotic behavior.  This simple argument remains (up to
some minor technicalities) the same also for infinite-dimensional systems
and has been applied for stability analysis, e.g.  in
\cite{Paz83,Dei92,CaH98}.  

On the other hand, converse Lyapunov theorems proved for wide classes of
infinite-dimensional systems show that asymptotic stability guarantees the
existence of a proper and positive definite Lyapunov function.

A first inkling that this is not the complete story comes from the study
of linear systems. In a seminal paper \cite{Dat70} Datko proved the
following. If $X$ is a Hilbert space and $A$ the generator of a $C_0$-semigroup on $X$, then the system $\dot x = A x$ is exponentially stable
if and only if there exists a positive definite bilinear form on $X$
(generated by a certain bounded positive definite linear operator $P$)
such that for all $x \in D(A)$ we have the following estimate for the scalar product of $Px, Ax$:
\begin{equation*}
    \langle Px , Ax \rangle < - \|x\|^2.
\end{equation*}
This is a natural extension of the finite-dimensional Lyapunov inequality.

At the same time the operator $P$ need not be coercive, so the natural Lyapunov function $V:x\mapsto \langle Px , x \rangle$ for the linear system $\dot x = Ax$ 
does not satisfy \eqref{eq:2}. In fact, there exist	 exponentially stable
$C_0$-semigroups on Hilbert spaces such that there does not exist an
equivalent scalar product under which the semigroup is a strict
contraction semigroup, \cite{Che76}. Hence, the non-coercivity of $P$
cannot be avoided in general.
In this situation, it appears that the left inequality in
\eqref{eq:2} is an artifact of the finite-dimensional origin of
the theory. In infinite dimensions it may sometimes be easier and more natural to derive
Lyapunov functions which have the weaker property that
\begin{equation}
    \label{eq:2b}
    0 < V(x) \leq \psi_2(\|x\|)\,,\quad x\neq 0.
\end{equation}
The next question then becomes under which decay conditions the existence
of such a non-coercive Lyapunov function guarantees uniform asymptotic stability. 

In this paper we provide a positive and a negative answer. We show that
for a large class of dynamical systems on Banach spaces, the existence of
a non-coercive Lyapunov function satisfying \eqref{eq:2b} and the decay
estimate \eqref{eq:3} guarantees global uniform asymptotic stability of
the origin provided a forward completeness property is satisfied and the
fixed point is robust. This is achieved using uniform Barbalat-like
estimates.  On the other hand, we show by examples that the results cannot
be extended much further. In particular, non-coercive Lyapunov functions
cannot be used in conjunction with a decay estimate in terms of the
Lyapunov function itself as in \eqref{eq:3b}. To show this, in Section~\ref{sec:examples} we give an
example of an exponentially unstable linear system with a bounded
generator admitting a Lyapunov function satisfying conditions
\eqref{eq:2b} and \eqref{eq:3b}.  Also it is a disadvantage that we have
to assume forward completeness whereas in many cases coercive Lyapunov
functions give this property for free. Example~\ref{example:notRFC} shows
that non-coercive Lyapunov functions do not guarantee forward completeness
even in the finite dimensional case.  Also they cannot be used to obtain bounds on the
transient behavior.

To compare this with a case of coercive Lyapunov functions, note that to
obtain boundedness of solutions on bounded intervals it is sufficient that
\begin{equation}
    \label{eq:3c}
    \dot V(x) < \alpha(\|x\|),
\end{equation}
where again we may appeal to \eqref{eq:2} to see that this is equivalent
to $\dot V(x) < \gamma(V(x))$ for a suitable function $\gamma$. In
\cite{AnS99} it is shown for a class of finite-dimensional systems that
forward completeness is equivalent to the existence of a coercive Lyapunov
function satisfying $\dot V(x) < V(x)$.  This argument is also applicable
to several classes of nonlinear infinite dimensional systems, see
\cite[Chapter 6]{Paz83}, \cite[Chapter 1]{KaJ11b}. As we will see in the
examples the assumption of coercivity cannot be relaxed if it is the aim
to conclude that solutions exist for all positive times from the fact that
solutions do not escape to infinity in finite-time.

If the origin is not a robust equilibrium, the existence of a non-coercive
Lyapunov function for a robustly forward complete system does not ensure
uniform global asymptotic stability of the origin. However, in this case
the origin still exhibits a form of attractivity.  Motivated by the
classical notion of weak attractivity \cite{Bha66, bhatia2002stability,
  SoW96} we introduce the stronger concept of uniform weak attractivity.
In the finite-dimensional case, this does not add anything new, but in
infinite dimensions uniform weak attractivity is essentially stronger than
weak attractivity.  
We show that the existence of a non-coercive Lyapunov function for a forward complete system
ensures that $0$ is uniformly globally weakly attractive.

In addition to the non-coercive Lyapunov theorem we also derive a converse
theorem for system which have flows with a Lipschitz continuity
property. The construction is motivated by classical converse theorems and
Yoshizawa's method \cite[Theorem 19.3]{Yos66}, \cite[Theorem
4.2.1]{Hen81}.  We derive in Section~\ref{sec:converse-theorems} a novel
"integral" construction of a non-coercive Lipschitz continuous Lyapunov
function for a UGAS system with a Lipschitz continuous flow map. The
construction appears to be novel, maybe because so far it has been unclear what the
interest in a non-coercive Lyapunov function is, in general.

The paper is organized as follows. In
  Section~\ref{sec:problem-statement} we introduce an abstract class of
  dynamical systems subject to time-varying disturbances and define the
  various stability concepts used in this paper.  In particular the
  concepts of robust forward completeness and robustness of the
  equilibrium point will be of importance and we introduce
  characterizations of these.  Section~\ref{sec:nonc-lyap-theor} is
  devoted to the derivation of sufficient conditions for uniform global
  and local asymptotic stability of the fixed point in terms of
  non-coercive Lyapunov functions. 
	In this section the assumptions of
  robust forward completeness and robustness of the equilibrium are vital. 
	However, we show that if the system fails to satisfy these additional properties, 
	then the origin is a global uniform weak attractor.
  In order to showcase the applicability of the results,
  Section~\ref{sec:Examples_I} discusses more concrete system
  classes which satisfy our general assumption.  In
  Section~\ref{sec:converse-theorems} we derive a converse theorem for
  systems which have a Lipschitz continuous flow.
  Section~\ref{sec:examples} discusses examples which give counterexamples
  showing that some of our results cannot be extended much
  further. We conclude in Section~\ref{sec:conclusions}.
	
	Some of results proved in this paper have been presented at the 10th IFAC Symposium on Nonlinear Control Systems (NOLCOS 2016) \cite{MiW16b}.

\subsection{Notation}

The following notation will be used throughout these notes. By $\R_+$ we
denote the set of nonnegative real numbers. For an arbitrary set $S$ and
$n\in\N$ the $n$-fold Cartesian product is $S^n:=S\times\ldots\times
S$. For normed linear spaces $X,Y$ we denote by $\mathcal{L}(X,Y)$ the space of
linear bounded operators acting from $X$ to $Y$ and we abbreviate $\mathcal{L}(X):=\mathcal{L}(X,X)$.
The open ball in a normed linear space $X$ with radius $r$ and center in $y \in X$ is denoted by $B_r(y):=\{x\in X \ |\  \|x-y\|_X<r\}$ (the space $X$ in which the ball is taken, will always be clear from the context). For short, we denote $B_r:=B_r(0)$. The (norm)-closure of a set $S \subset X$ will be denoted by $\overline{S}$.

For the formulation of
stability properties the following classes of comparison functions are
useful, see \cite{Hah67,Kel14}:
\begin{equation*}
\begin{array}{ll}
{\K} &:= \left\{\gamma:\R_+ \to \R_+ \left|\ \right. \gamma\mbox{ is continuous and strictly increasing, }\gamma(0)=0\right\},\\
{\K_{\infty}}&:=\left\{\gamma\in\K\left|\ \gamma\mbox{ is unbounded}\right.\right\},\\
{\LL}&:=\left\{\gamma:\R_+ \to \R_+ \left|\ \gamma\mbox{ is continuous and decreasing with}\right.
 \lim\limits_{t\rightarrow\infty}\gamma(t)=0 \right\},\\
{\KL} &:= \left\{\beta: \R_+^2 \to \R_+ \left|\ \right. \beta(\cdot,t)\in{\K},\ \forall t \geq 0,\  \beta(r,\cdot)\in {\LL},\ \forall r >0\right\}.
\end{array}
\end{equation*}

\section{Problem statement}
\label{sec:problem-statement}

In this paper we consider abstract axiomatically defined time-invariant
and forward complete systems on the state space $X$ which are subject to a shift-invariant set of
disturbances $\Dc$.
\begin{definition}
\label{Steurungssystem}
Consider the triple $\Sigma=(X,\Dc,\phi)$, consisting of 
\begin{enumerate}[(i)]  
	\item A normed linear space $(X,\|\cdot\|_X)$, called the {state space}, endowed with the norm $\|\cdot\|_X$.
	\item A set of  disturbance values $D$, which is a nonempty subset of a normed linear space $S_d$.
	\item A {space of disturbances} $\Dc \subset \{d:\R_+ \to D\}$
          satisfying the following two axioms.
					
\textit{The axiom of shift invariance} states that for all $d \in \Dc$ and all $\tau\geq0$ the time
shift $d(\cdot + \tau)$ is in $\Dc$.

\textit{The axiom of concatenation} is defined by the requirement that for all $d_1,d_2 \in \Dc$ and for all $t>0$ the concatenation of $d_1$ and $d_2$ at time $t$
\begin{equation}
d(\tau) := 
\begin{cases}
d_1(\tau), & \text{ if } \tau \in [0,t], \\ 
d_2(\tau-t),  & \text{ otherwise},
\end{cases}
\label{eq:Composed_Input}
\end{equation}
belongs to $\Dc$.

	\item A transition map $\phi:\R_+ \times X \times \Dc \to X$.
\end{enumerate}
The triple $\Sigma$ is called a (forward complete) dynamical system, if the following properties hold:

\begin{sysnum}
	\item forward completeness: for every $(x,d) \in X \times \Dc$ and
          for all $t \geq 0$ the value $\phi(t,x,d) \in X$ is well-defined.
	\item\label{axiom:Identity} The identity property: for every $(x,d) \in X \times \Dc$
          it holds that $\phi(0, x,d)=x$.
	\item Causality: for every $(t,x,d) \in \R_+ \times X \times
          \Dc$, for every $\tilde{d} \in \Dc$, such that $d(s) =
          \tilde{d}(s)$, $s \in [0,t]$ it holds that $\phi(t,x,d) = \phi(t,x,\tilde{d})$.
	\item \label{axiom:Continuity} Continuity: for each $(x,d) \in X \times \Dc$ the map $t \mapsto \phi(t,x,d)$ is continuous.
		\item \label{axiom:Cocycle} The cocycle property: for all $t,h \geq 0$, for all
                  $x \in X$, $d \in \Dc$ we have
$\phi(h,\phi(t,x,d),d(t+\cdot))=\phi(t+h,x,d)$.
\end{sysnum}
\end{definition}
Here  $\phi(t,x,d)$ denotes the state of a system at the moment $t \in
\R_+$ corresponding to the initial condition $x \in X$ and the disturbance $d \in \Dc$.

\begin{definition}
\label{axiom:Lipschitz}
We say that the flow of $\Sigma=(X,\Dc,\phi)$ is Lipschitz continuous on compact intervals, if 
for any $\tau>0$ and any $r>0$ there exists $L>0$ so that 
\begin{equation}
x,y\in \overline{B_r},\ t \in [0,\tau], d\in\Dc \qrq \|\phi(t,x,d) - \phi(t,y,d) \|_X \leq L \|x-y\|_X.
\label{eq:Flow_is_Lipschitz}
\end{equation}	
\end{definition}

We exploit the following stronger version of forward completeness:
\begin{definition}
\label{Def_RFC}
The system $\Sigma=(X,\Dc,\phi)$ is called robustly forward complete (RFC) 
if for any $C>0$ and any $\tau>0$ it holds that 
\[
\sup\big\{ \|\phi(t,x,d)\|_X \ |\  \|x\|_X\leq C,\: t \in [0,\tau],\: d\in \Dc \big\} < \infty.
\]
\end{definition}

The condition of robust forward completeness is satisfied by large classes of infinite-dimensional
systems. 

\begin{definition}
\label{def:EqPoint}
We call $0 \in X$ an equilibrium point of the system $\Sigma=(X,\Dc,\phi)$, if 
$\phi(t,0,d) = 0$ for all $t \geq 0$ and all $d \in \Dc$.
\end{definition}

\begin{definition}
\label{def:RobustEqPoint}
We call $0 \in X$ a robust equilibrium point of the system
$\Sigma=(X,\Dc,\phi)$, if it is an equilibrium point such that 
for every $\eps >0$ and for any $h>0$ there exists $\delta = \delta
(\eps,h)>0$, so that 
\begin{equation}
\hspace{-7mm} t\in[0,h],\ \|x\|_X \leq \delta, \ d \in \Dc \quad \Rightarrow \quad  \|\phi(t,x,d)\|_X \leq \eps.
\label{eq:RobEqPoint}
\end{equation}
\end{definition}

\begin{example}
\label{exam:Relationships_RFC_REP_etc}
Let $X = D = \R$ and let $\mathcal{D}= L^\infty(\R_+,D)$. The following
examples show the relations between forward completeness, robust
forward completeness and robustness of the equilibrium point.

\begin{enumerate}[(i)]  
	\item $\Sigma$ is RFC, but 0 is not a REP of $\Sigma$: \quad $\dot{x} = |d|(x - x^3).$
\item $\Sigma$ is forward complete, but not RFC and $0$ is not a REP: \quad $\dot{x} = d \cdot x.$
\item $0$ is a REP of $\Sigma$, $\Sigma$ is forward complete, but not RFC:
\[
\dot{x} = \frac{1}{|d|+1} x + d \max\big\{|x|-1,0\big\}.
\]	
\item $0$ is a REP of $\Sigma$ and $\Sigma$ is RFC: \quad $\dot{x} = \frac{1}{|d|+1} x.$
\end{enumerate}

\end{example}

\begin{lemma}
\label{lem:RobustEquilibriumPoint}
Let $\Sigma=(X,\Dc,\phi)$ be a  system with a flow which is Lipschitz continuous on compact intervals.
If $0 \in X$ is an equilibrium point of $\Sigma$, then $0$ is a robust equilibrium point of $\Sigma$.
\end{lemma}

\begin{proof}
  Clear as the robustness property of an equilibrium is the continuity of the function $\xi:(x,h)\mapsto \sup_{t\in[0,h],\ d\in\Dc}\|\phi(t,x,d)\|_X$ with respect to its first argument at $x=0$.
\end{proof}

In this paper we investigate the following stability properties of equilibria of
abstract systems.

\begin{definition}{}
    \label{d:stability_new}
Consider a system $\Sigma=(X,\Dc,\phi)$ with fixed point $0$. The
equilibrium position $0$ is called
\begin{enumerate}[(i)]
  \item (locally) uniformly stable (US), if for every $\eps >0$ there is a $\delta>0$ so that
\begin{eqnarray}
\|x\|_X \leq \delta,\ d\in\Dc,\ t \geq 0 \quad\Rightarrow\quad\|\phi(t,x,d)\|_X \leq \eps.
\label{eq:Uniform_Stability}
\end{eqnarray}
\item uniformly globally asymptotically
stable (UGAS)
if there exists a $\beta \in \KL$ such that for all $x \in X, d\in \Dc, t \geq 0$
\begin{equation}
\label{UGAS_wrt_D_estimate}
\hspace{-1mm}\|\phi(t,x,d)\|_X \leq \beta(\|x\|_X,t).
\end{equation}
\item uniformly (locally) asymptotically stable (UAS)
if there exist a $\beta \in \KL$ and an $r>0$ such that 
the inequality \eqref{UGAS_wrt_D_estimate} holds for all $x \in B_r,d\in \Dc,t \geq 0$.

\item globally weakly attractive, if
\begin{eqnarray}
x\in X,\ d \in\Dc \quad\Rightarrow\quad \inf_{t \geq 0} \|\phi(t,x,d)\|_X = 0.
\label{eq:LIM_inf_form}
\end{eqnarray}

\item \label{def:UniformGlobalWeakAttractivity} uniformly globally weakly attractive, if for every $\eps>0$ and
for every $r>0$ there exists a $\tau = \tau(\eps,r)$ such that 
\begin{eqnarray}
\|x\|_X\leq r,\ d\in\Dc \qrq \exists t = t(x,d,\eps) \leq \tau:\ \|\phi(t,x,d)\|_X \leq \eps.
\label{eq:Uniform_weak2}
\end{eqnarray}
\item \label{def:UniformGlobalAttractivity} uniformly globally attractive (UGATT), if for any $r,\eps >0$ there exists $\tau=\tau(r,\eps)$ so that
\begin{eqnarray}
\|x\|_X \leq r,\ d\in\Dc,\ t \geq \tau(r,\eps) \quad \Rightarrow \quad \|\phi(t,x,d)\|_X \leq \eps.
\label{eq:UAG_with_zero_gain}
\end{eqnarray}
\end{enumerate}
\end{definition}

\begin{remark}
According to Definition~\ref{d:stability_new} the origin is globally weakly attractive  if and only if
for all $x\in X,\ d \in\Dc$ the origin belongs to the $\omega$-limit set $\omega(x,d)$, defined by
\begin{equation}
    \label{eq:eq:weakatt}
    \omega(x,d) := \{ y \in X \;|\; \exists \ t_k \to \infty : \phi(t_k,x,d) \to y \}.
\end{equation}
In this form $0$ is a weak attractor, a notion originally introduced in
\cite{Bha66}, see also \cite{bhatia2002stability} and our nomenclature
follows this reference.

On the other hand, in the language of \cite{SoW96}, weak attractivity is the limit property with zero
gain. Clearly, $0$ is a global weak attractor of $\Sigma$ if and only if
for every $x\in X$, every $d \in \Dc$ and any $\eps>0$ there exists
$t=t(x,d,\eps)$ so that $\|\phi(t,x,d)\|_X \leq \eps$. This justifies the
interpretation of \eqref{def:UniformGlobalWeakAttractivity} as a uniform version of weak attractivity.

 We stress that in some works "weak" stability
    concepts are discussed which are related to convergence properties in
    the weak topology on $X$, e.g. \cite{eisner2007weakly}. This is not the intended meaning here.
\end{remark}

The following characterization of UGAS, \cite[Theorem 2.2]{KaJ11b},
shows that the concept of uniform global attractivity is very 
helpful in the verification of UGAS.

\begin{proposition}
\label{prop:UGAS_Characterization}
Let $\Sigma=(X,\Dc,\phi)$ be a control system and let $0$ be a robust equilibrium point for $\Sigma$.
Then $\Sigmafp$ is UGAS if and only if $\Sigma$ is robustly forward complete and uniformly globally attractive.
\end{proposition}

It may be surprising at first glance that the characterization of UGAS
does not directly require a stability property, whereas in the usual
context of ODEs it is well known that attractivity on its own does not imply
asymptotic stability. The point to notice here is that uniform
attractivity is a far stronger concept than attractivity as it requires
convergence rates that are uniform for all initial conditions from a ball
around the fixed point. 
 This ensures that uniformly attractive systems, having 0 as a
  robust equilibrium point are uniformly stable. Note however, that
  uniform attractivity without robustness of the equilibrium point does
  not imply stability anymore. This is illustrated by the following example.
\begin{example}
\label{exam:UGATT_FC_not_REP_not_RFC}
Let $\Dc :=L^\infty(\R_+,\R)$, $x,y \in\R$ and consider
\begin{subequations}
\begin{align}
\dot{x}(t) =&\ d(t) x(t) y(t) - x^3(t) - x^{1/3}(t), \\
\dot{y}(t) =&\ - y^3(t) - y^{1/3}(t).
\end{align}
\label{eq:Counterexample}
\end{subequations}
It is easy to see that \eqref{eq:Counterexample} is forward complete and $0$ is its equilibrium.
This system is also UGATT in $x^* = 0$.  To see this, note that there exists a time
$t^* > 0$, so that for any initial condition $y_0$ the solution of the
$y$-subsystem attains the value $0$ in time less or equal than $t^*$.
Hence, $d(t) x(t) y(t) = 0 $ for $t\geq t^*$,  any $d\in\Dc$ and any initial conditions $(x_0,y_0)\in\R^2$.
Consequently, the first component of the solution $\phi_1(t,(x_0,y_0),d)=0$  
for all $t\geq 2t^*$ regardless of the initial condition. In particular,
\eqref{eq:Counterexample} is UGATT in $x^* = 0$.

On the other hand, it is easy to see that for any initial condition
$z_0:=(x_0,y_0)$ with $x_0>0$, $y_0>0$ and for any $K>0$, and any small enough $\tau>0$ one can find $d \in\Dc$ so that $|\phi(\tau,z_0,d)|>K$.
In particular, this shows that zero is not a robust equilibrium of
\eqref{eq:Counterexample} and that \eqref{eq:Counterexample} is not
RFC. The strict analysis of this example is straightforward, and we skip
it. ~ \hfill{} $\square$
\end{example}

\subsection{Characterizations of Robust Forward Completeness}

In this preliminary section we provide useful restatements in terms of the comparison
functions of robust forward completeness and of robust equilibrium points. 

We call
a function $h: \R_+^2 \to \R_+$ increasing, if $(r_1,R_1) \leq (r_2,R_2)$
implies that $h(r_1,R_1) \leq h(r_2,R_2)$, where we use the component-wise
partial order on $\R_+^2$. We call $h$ strictly increasing if $(r_1,R_1)
\leq (r_2,R_2)$ and $(r_1,R_1) \neq (r_2,R_2)$ imply $h(r_1,R_1) <
h(r_2,R_2)$.
\begin{lemma}
\label{lem:RFC_criterion}
Consider a  forward complete system $\Sigma=(X,\Dc,\phi)$. The following statements are equivalent:
\begin{enumerate}
	\item[(i)] $\Sigma$ is robustly forward complete.
	\item[(ii)] there exists a continuous, increasing function $\mu: \R_+^2 \to \R_+$, such that
 \begin{equation}
    \label{eq:8}
    x\in X,\ d\in \Dc,\ t \geq 0 \qrq \| \phi(t,x,d) \|_X \leq \mu( \|x\|_X,t).
\end{equation}	
	\item[(iii)] there exists a continuous function $\mu: \R_+^2 \to \R_+$ such that the implication \eqref{eq:8} holds.
\end{enumerate}
\end{lemma}

\begin{proof}
(i) $\Rightarrow$ (ii). Let $\Sigma$ be robustly forward complete. Define $\tilde\mu: \R_+^2 \to \R_+$ by
\begin{eqnarray}
\label{eq:tildechi-def}
\tilde\mu(C,\tau) := \sup\{ \|\phi(t,x,d)\|_X
\ |\  \|x\|_X\leq C,\: t \in [0,\tau],\: d\in \Dc \}
\end{eqnarray}
which is well-defined due to the robust forward completeness of $\Sigma$.
Clearly, $\tilde\mu$ is increasing by definition. In particular, it is
locally integrable. 

Now define $\hat\mu: (0,+\infty)^2 \to \R_+$ by setting for $C>0$ and $\tau>0$ 
\begin{eqnarray}
\label{eq:chi-def}
\hat\mu(C,\tau) := \frac{1}{C\tau} \int_C^{2C}\int_{\tau}^{2\tau} \tilde\mu(r,s) drds + C\tau.
\end{eqnarray}
By construction, $\hat\mu$ is strictly increasing and continuous on $(0,+\infty)\times(0,+\infty)$.
We can also enlarge the domain of definition of $\hat\mu$ to all of $\R^2_+$
using monotonicity. To this end we define for $\tau>0$: $\hat\mu(0,\tau) := \lim_{C\to +0}
\hat\mu(C,\tau)$ and for $C\geq 0$ we define $\hat\mu(C,0) := \lim_{\tau\to +0}
\hat\mu(C,\tau)$. This is well-defined as $\hat\mu$ is increasing on
$(0,+\infty)^2$ and we obtain that the resulting function is increasing on
$\R_+^2$. Note that the construction does not guarantee that $\hat\mu$ is
continuous. In order to obtain continuity choose a continuous, strictly
increasing function $\nu: \R_+ \to \R_+$ with $\nu(r) > \max\{ \hat\mu(0,r),\hat\mu(r,0)\}$, $r \geq 0$ and
define 
for $(C,\tau) \geq (0,0)$
\begin{eqnarray} \quad
\mu(C,\tau) := 
      \max\big\{
  \nu\big(\max\{C,\tau\}\big),\hat\mu(C,\tau) \big\} + C\tau.
\label{eq:mu_def}
\end{eqnarray}
It is easy to see that $\mu$ is continuous as $\mu(C,\tau) =
\nu\big(\max\{C,\tau\}\big) + C\tau$ whenever $C$ or $\tau$ is small
enough.
At the same time we have for $C>0,\tau >0$ that 
\begin{eqnarray*}
\mu(C,\tau) \geq \hat\mu(C,\tau) \geq \frac{1}{C\tau}
\int_C^{2C}\int_{\tau}^{2\tau}  drds \ \tilde\mu(C,\tau) + C\tau \geq \tilde\mu(C,\tau).
\end{eqnarray*}
This implies that (ii) holds with this $\mu$.

(ii) $\Rightarrow$ (iii) is evident.

(iii) $\Rightarrow$ (i) follows due to continuity of $\mu$.
\end{proof}

\begin{remark}
\label{rem:Lemma_Karafyllis}
  Lemma~\ref{lem:RFC_criterion} follows from \cite[Lemma 3.5]{Kar04}, but we included the proof of Lemma~\ref{lem:RFC_criterion} since it is used 
	for the proof of Proposition~\ref{prop:RFC_and_0-REP_criterion}.
\end{remark}

\begin{lemma}
\label{lem:0-REP_criterion}
Consider a forward complete system $\Sigma=(X,\Dc,\phi)$. The following statements are equivalent:
\begin{enumerate}
	\item[(i)] 0 is a robust equilibrium point of $\Sigma$.
	\item[(ii)] For any $\tau \geq 0$ there exists a $\delta=\delta(\tau)>0$
           such
          that the function 
$\tilde\mu(\cdot,\tau):[0,\delta) \to \R_+$
defined by
\begin{equation}
\tilde\mu(r,\tau) :=
\sup\{ \|\phi(t,x,d)\|_X \ |\  \|x\|_X\leq r,\: t \in [0,\tau],\: d\in \Dc \}\,,\quad r \in [0,\delta)
\label{eq:omega_h}
\end{equation}	
is continuous at $r=0$ with $\tilde\mu(0,\tau)=0$.
\end{enumerate}
\end{lemma}

\begin{proof}
    (i) $\Rightarrow$ (ii).  Let $0$ be a robust equilibrium point of
    $\Sigma$. 
 For $\tau=0$ we have $\tilde\mu(r,0)=r$ by ($\Sigma$2) and there is
   nothing to show. By the assumption, for any $\tau>0$ the function $\tilde\mu(\cdot,\tau)$ defined by
    \eqref{eq:omega_h} is well-defined
    on an interval $[0,\delta(\tau))$ with $\delta(\tau)>0$. In addition,
    $\tilde\mu(0,\tau)=0$ for any $\tau>0$.
Assume (ii) does not
hold. Then there exists a $\tau^*>0$ so that the corresponding function 
$\tilde\mu(\cdot,\tau^*)$,  is not continuous at zero.

According to the definition, $\tilde\mu(\cdot,\tau^*)$ is nondecreasing on its domain
of definition. Therefore $a:=\lim_{r \to +0}\tilde\mu(r,\tau^*)$ exists.  Then
there exist sequences $\{t_k\}_{k\geq 1} \subset [0,\tau]$, $\{x_k\}_{k\geq
  1} \subset X$ with  $\|x_k\|_X \to 0$ as $k \to \infty$, and
$\{d_k\}_{k\geq 1} \subset \Dc$  so that
\[
\|\phi(t_k,x_k,d_k)\|_X \geq \frac{a}{2} \quad \forall k \geq 1.
\]
This contradicts the fact that $0$ is a robust equilibrium point of $\Sigma$.

(ii) $\Rightarrow$ (i).
Pick any $\tau>0$ and any $\eps>0$. Due to the continuity of
  $\tilde\mu(\cdot,\tau)$ at $r=0$ there
exists a $\delta>0$ so that 
\[
\sup\big\{ \|\phi(t,x,d)\|_X
\ |\  \|x\|_X\leq \delta,\: t \in [0,\tau],\: d\in \Dc \big\} <\eps,
\]
which shows that $0$ is a robust equilibrium point of $\Sigma$.
\end{proof}

\begin{proposition}
\label{prop:RFC_and_0-REP_criterion}
Consider a forward complete system $\Sigma=(X,\Dc,\phi)$. The following statements are equivalent:
\begin{enumerate}
	\item[(i)] $\Sigma$ is robustly forward complete and $0$ is a robust equilibrium point of $\Sigma$.
	\item[(ii)] there exists a continuous, radially unbounded function $\mu: \R_+^2 \to
          \R_+$ such that 
	$\mu(\cdot,h) \in\Kinf$ for all $h \geq 0$ and the implication \eqref{eq:8} holds.
	\item[(iii)] there exist $\sigma \in \Kinf$ and a continuous function $\chi: \R_+^2 \to
    \R_+$ such that $\chi(r,\cdot) \in \K$ for all $r > 0$, $\chi(0,t)=0$ for all $t\in \R_+$ and
such that for all $x\in X, d\in \mathcal{D}$ and all $t \geq 0$ we have   
\begin{equation}
    \label{eq:RFC_and_REP}
    \| \phi(t,x,d) \|_X \leq \sigma(\|x\|_X) + \chi( \|x\|_X,t).
\end{equation}	
\end{enumerate}
\end{proposition}

\begin{proof}
    (i) $\Rightarrow$ (ii).  Since $\Sigma$ is RFC,
    Lemma~\ref{lem:RFC_criterion} implies that there exists a continuous,
    increasing function $\mu: \R_+^2 \to \R_+$ such that the implication \eqref{eq:8} holds.  It
    remains to show that we may choose $\mu$ such that $\mu(0,h) = 0$ for
    all $h \geq 0$.

Consider the function $\tilde\mu: \R_+^2 \to \R_+$ defined by \eqref{eq:tildechi-def}.
Since $0$ is a robust equilibrium point of $\Sigma$, 
we have by Lemma~\ref{lem:0-REP_criterion} for all
  $\tau \geq 0$ that 
  $\tilde\mu(0,\tau)=0$ and $r=0$ is a point of continuity of $\tilde\mu(\cdot,\tau)$.
For $\hat\mu: \R_+^2 \to \R_+$ defined by \eqref{eq:chi-def}
it thus holds for all $C>0,h>0$ that
\[
\hat\mu(C,h) \leq \tilde\mu(2C,2h) + Ch \to 0, \text{ as } C \to +0,
\]
which shows that $\hat\mu(0,h)= 0$ for every $h\geq 0$  and 
$r=0$ is a point of continuity of $\hat\mu(\cdot,h)$ for any $h\geq 0$.

Consider $\nu: \R_+ \to \R_+$ from the proof of
Lemma~\ref{lem:RFC_criterion}, which was chosen such that $\nu(r) > \max\{
\hat\mu(0,r),\hat\mu(r,0)\} = \hat\mu(r,0)$ for all $r>0$. 
  Since, by Lemma~\ref{lem:0-REP_criterion},  
$\hat\mu(r,0) \to 0$ as $r\to +0$, we may choose $\nu$ with $\nu(0)=0$, and thus $\nu\in\K$.

Now consider the following function $\mu: \R_+^2 \to \R_+$ (which is different from that of \eqref{eq:mu_def}):
\begin{eqnarray} \quad
\mu(C,\tau) := 
      \max\big\{
  \nu\big(C\big),\hat\mu(C,\tau) \big\} + C\tau.
\label{eq:mu_def_2}
\end{eqnarray}
Since $\nu$ is continuous on $\R_+$ and $\hat\mu$ is continuous over $[0,+\infty)\times (0,+\infty)$, 
$\mu$ is also continuous over $[0,+\infty)\times (0,+\infty)$.
At the same time for $\tau$ small enough $\mu(C,\tau)=\nu(C) + C\tau$, and
again it is not hard to show that $\mu$ is continuous over $\R^2_+$.
Moreover, it is clear that $\mu$ satisfies item (ii) of Lemma~\ref{lem:RFC_criterion}.

It follows from the definition that $\mu(0,h)=0$. Since $\hat\mu(C,\tau)
\geq C$ for any $C\geq 0$ and $\tau\geq 0$ (due to the identity axiom
($\Sigma$\ref{axiom:Identity})), and since both $\nu$ and $\hat\mu$ are increasing, $\mu(\cdot,h)\in\Kinf$ for any $h>0$.

(ii) $\Rightarrow$ (iii).
Let $\mu$ as in (ii) exist. Then for all $r \in \R_+$ and all $t \in \R_+$ we have that
\[
\mu(r,t) = \underbrace{\mu(r,0)}_{\sigma(r)} + \underbrace{\mu(r,t) - \mu(r,0)}_{\chi(r,t)}.
\]
Now (iii) is satisfied with these $\sigma$ and $\chi$.

(iii) $\Rightarrow$ (i). Let $\sigma$ and $\mu$ be as in (iii). 
Then $\phi(t,0,d) = 0$ for all $t\geq 0$ and $d\in \Dc$. Now pick any $\eps>0$ and any $\tau >0$. Set $\delta_1:=\sigma^{-1}(\frac{\eps}{2})$ and choose $\delta_2$ so that 
\[
\sup_{0\leq r \leq \delta_2,\ 0\leq t \leq \tau} \chi(r,t) = \sup_{0\leq r \leq \delta_2} \chi(r,\tau) \leq \frac{\eps}{2}.
\]
Now define $\delta:=\min\{\delta_1,\delta_2\}$.
Then
\[
t\in[0,\tau],\ \|x\|_X\leq \delta,\ d\in\Dc \qrq \|\phi(t,x,d)\|_X \leq \eps,
\]
which shows that 0 is a robust equilibrium point of $\Sigma$.

Obviously, (iii) implies robust forward completeness of $\Sigma$.
\end{proof}

\begin{remark}
In order to ensure the existence of $\sigma \in \Kinf$ and $\chi$ as in
item (iii) of Proposition~\ref{prop:RFC_and_0-REP_criterion}, it is not
sufficient to assume that $\Sigma$ is RFC and $0$ is an equilibrium
point. Indeed, for the system from item (i) in Example~\ref{exam:Relationships_RFC_REP_etc}
the function $\tilde\mu$ from \eqref{eq:tildechi-def} can be computed for all 
$C>0$ and $\tau>0$ as $\tilde\mu(C,\tau) = \max\{1,C\}$ and for all $C\geq 0$, $\tau \geq 0$ as $\tilde\mu(C,0) = C$, $\tilde\mu(0,\tau) = 0$. Clearly, one cannot 
majorize this function by any functions $\sigma \in \Kinf$ and $\chi$ as in item (iii) of Proposition~\ref{prop:RFC_and_0-REP_criterion}.
\end{remark}

\begin{remark}
Setting $t:=0$ in \eqref{eq:RFC_and_REP} and using the identity axiom ($\Sigma$\ref{axiom:Identity}) we see that for 
$\sigma\in\Kinf$ in \eqref{eq:RFC_and_REP} it holds that $\sigma(r)\geq r$ for all $r\geq0$. For some systems it is possible to choose  $\sigma(r) := r$ for all $r \in\R_+$, but in general such a choice is not possible.
Consider a linear system  $\dot{x} = Ax$, where $A$ is the generator of a $C_0$-semigroup $T(\cdot)$ over $X$, 
satisfying $\lim_{t\to+0}\|T(t)\|_{\mathcal{L}(X)} >1$ (there are many examples of such semigroups).
Then there exist a sequence $\{x_k\} \subset X$: $\|x_k\|_X =1$ for all $k \in\N$ and a corresponding sequence of times $t_k \to +0$ as $k\to\infty$ so that $\|T(t_k)x_k\|_X>1+\eps$ for some $\eps>0$. Hence we have
\[
1+\eps < \lim_{k\to \infty} \|T(t_k)x_k\|_X  \leq \lim_{k\to \infty} \big( \sigma(\|x_k\|_X) + \chi(\|x_k\|_X,t_k) \big) =\sigma(1).
\]
\end{remark}

\section{Non-coercive Lyapunov theorems}
\label{sec:nonc-lyap-theor}

Lyapunov functions provide a predominant tool to study UGAS. In our
context they are defined as follows. Recall that for a continuous function
$h: \R \to \R$ the (right-hand lower) Dini derivative at a point $t$ is
defined by, see \cite{szarski1965differential},
\begin{equation}
    \label{eq:4}
    D_+h(t):=\mathop{\underline{\lim}} \limits_{\tau \rightarrow +0} {\frac{1}{\tau}\big(h(t+\tau)-h(t)\big) }.
\end{equation}

Consider a system $\Sigma= (X,\Dc,\phi)$ and let $V:X \to\R$ be continuous.
Given $x\in X,d\in \mathcal{D}$, we consider the (right-hand lower) Dini
derivative of the continuous function $t \mapsto V(\phi(t,x,d))$ at $t=0$:
\begin{equation}
\label{UGAS_wrt_D_LyapAbleitung}
\dot{V}_d(x):=\mathop{\underline{\lim}} \limits_{t \rightarrow +0} {\frac{1}{t}\big(V(\phi(t,x,d))-V(x)\big) }.
\end{equation}
We call this the Dini derivative of $V$ along the trajectories of $\Sigma$.

\begin{definition}
\label{def:UGAS_LF_With_Disturbances}
A continuous function $V:X \to \R_+$ is called a \textit{Lyapunov function} for system $\Sigma=(X,\Dc,\phi)$,  if there exist
$\psi_1,\psi_2 \in \Kinf$ and $\alpha \in \K$
such that 
\begin{equation}
\label{LyapFunk_1Eig_UGAS}
\psi_1(\|x\|_X) \leq V(x) \leq \psi_2(\|x\|_X) \quad \forall x \in X
\end{equation}
holds 
and the Dini derivative of $V$  along the trajectories of $\Sigma$ satisfies 
\begin{equation}
\label{DissipationIneq_UGAS_With_Disturbances}
\dot{V}_d(x) \leq -\alpha(\|x\|_X)
\end{equation}
for all $x \in X$ and all $d \in \Dc$. 
We call $V$ a \textit{non-coercive Lyapunov function}, if instead of
\eqref{LyapFunk_1Eig_UGAS} we have $V(0)=0$ and
\begin{equation}
    \label{eq:1}
    0 < V(x) \leq \psi_2(\|x\|_X) \quad \forall x \in X \backslash\{0\}.
\end{equation}
\end{definition}

If we want to emphasize that \eqref{LyapFunk_1Eig_UGAS} holds we will also
speak of a coercive Lyapunov function. 
The following result is well-known:
\begin{proposition}
\label{Direct_LT_0-UGAS_maxType}
Let $\Sigma=(X,\Dc,\phi)$ be a dynamical system. If there exists a coercive Lyapunov function for $\Sigma$, then $\Sigmafp$ is UGAS.
\end{proposition}

The proof of Proposition~\ref{Direct_LT_0-UGAS_maxType}  is analogous to the proof of its finite-dimensional counterpart, see \cite[p. 160]{LSW96}. Note however, that 
we use continuous Lyapunov functions and the trajectories of the system $\Sigma$ are merely continuous,
therefore we cannot use the standard comparison principle, see \cite[Lemma 4.4]{LSW96} in the proof of 
Proposition~\ref{Direct_LT_0-UGAS_maxType}. Instead one can exploit the
following generalized comparison principle from \cite[Lemma 6.1]{szarski1965differential}, \cite[Lemma 1]{MiI16}:
\begin{lemma}
\label{thm:ComparisonPrinciple}
Let $\alpha \in \mathcal{P}$  and consider the differential inequality
\begin{equation}
\dot{y}(t) \leq -\alpha(y(t))\,, \quad t>0.
\label{eq:ComparisonPrinciple}
\end{equation}
There exists a $\beta \in \KL$ so that for all continuous functions $y:
\R_+ \to \R_+$ satisfying \eqref{eq:ComparisonPrinciple} in the sense of
the Dini
derivative (defined in \eqref{UGAS_wrt_D_LyapAbleitung}) we have
\begin{equation}
y(t) \leq \beta(y(0),t) \quad \forall t \geq 0.
\label{eq:ComparisonPrinciple_FinalEstimate}
\end{equation}
\end{lemma}

Next we show that already the existence of a non-coercive Lyapunov
function is sufficient for UGAS of a system provided another mild
assumption is satisfied. To this end we need the following property of
Dini derivatives. 
  \begin{lemma}
\label{Dinilemma}
      Let $f,g: [0,\infty) \to \R$ be continuous. If for all $t\geq 0$ we
      have $D_+ f(t) \leq - g(t)$, then for all $t\geq 0$ it follows that
\begin{equation}
    \label{eq:33}
    f(t) - f(0) \leq - G(t) := - \int_0^t g(s) ds.
\end{equation}
  \end{lemma}

  \begin{proof}
      As $g$ is continuous, it follows that $D_+ \left(f(t) + G(t)\right)
      = (D_+ f(t)) + g(t) \leq 0$
      for all $t\geq 0$. It follows from \cite[Theorem
      2.1]{szarski1965differential} that $f+G$ is decreasing. As $G(0) = 0$ the claim follows. 
  \end{proof}
Alternative arguments for this simple property may be found in
\cite[pp. 204-205]{Sak47}, \cite{HaT06}.

\begin{theorem}{\textbf{(Non-coercive UGAS Lyapunov theorem)}}
    \label{t:noncoeLT}
    Consider a system $\Sigma=(X,\Dc,\phi)$ and assume that $\Sigma$ is robustly forward complete and $0$ is a robust equilibrium of $\Sigma$.
		If $V$ is a non-coercive Lyapunov function for $\Sigma$, then $\Sigmafp$ is UGAS.
\end{theorem}

\begin{proof}
Let $V$ be a non-coercive Lyapunov function and let $\alpha \in \mathcal{K}$ be such that we have the decay estimate \eqref{DissipationIneq_UGAS_With_Disturbances}.
Along any trajectory $\phi$ of $\Sigma$ 
we have the inequality
\begin{equation}
\dot V_{d(t+\cdot)}(\phi(t,x,d)) \leq -\alpha(\|\phi(t,x,d)\|_X), \quad
\forall t\geq 0.
\label{eq:JustEq1}
\end{equation}
It follows from Lemma~\ref{Dinilemma}  that 
\begin{equation}
V(\phi(t,x,d)) - V(x) \leq -\int_0^t \alpha(\|\phi(s,x,d)\|_X)ds,
\label{eq:JustEq3}
\end{equation}
which implies that for all $t\geq 0$ we have
\begin{eqnarray}
\int_0^t \alpha(\|\phi(s,x,d)\|_X)ds  
&\leq&
V(x).
\label{eq:Vintbound}
\end{eqnarray}

\textbf{Step 1: (Stability)} 
 Seeking a contradiction, assume that $\Sigma$ is not uniformly stable in $x^* = 0$. Then there exist an
$\varepsilon >0$ and sequences $\{ x_k \}_{k\in\N}$ in $X$, $\{ d_k
\}_{k\in\N}$ in $\mathcal{D}$, and $t_k \geq 0$ such that $x_k \to 0$ as $k \to \infty$ and
\begin{equation*}
    \| \phi(t_k,x_k,d_k) \|_X = \varepsilon \quad \forall k \geq 1.
\end{equation*}
By the bound on $V$ given by \eqref{eq:1} it follows that $V(x_k)\to
0$. 

Since $\Sigma$ is RFC and $0$ is a robust equilibrium point of $\Sigma$, Proposition~\ref{prop:RFC_and_0-REP_criterion}
implies that there exist $\sigma \in \Kinf$ and $\chi$ as in item (iii) of Proposition~\ref{prop:RFC_and_0-REP_criterion} so that \eqref{eq:RFC_and_REP}
holds.

Appealing to continuity of $\chi$ we may choose $\tau>0$ such that $\chi(r,\tau) \leq \varepsilon/2$ for all $0\leq r \leq \varepsilon$.

 Using \eqref{eq:RFC_and_REP} we obtain that for all $k \in \N$ and for all $t \in [t_k-\tau,t_k ]$ we have either
$\|\phi(t,x_k,d_k) \|_X > \varepsilon$ or
\begin{multline*}
    \sigma\big(\| \phi(t,x_k,d_k) \|_X\big) \geq \| \phi(t_k,x_k,d_k) \|_X - \chi\big(\|
    \phi(t,x_k,d_k) \|_X, t_k -t\big)  \geq \varepsilon -
    \frac{\varepsilon}{2} = \frac{\varepsilon}{2}.
\end{multline*}
Setting $t:=0$ in \eqref{eq:RFC_and_REP} and using the identity axiom ($\Sigma$\ref{axiom:Identity}) we see that $\sigma(r)\geq r$ for all $r\geq0$, and thus $\sigma^{-1}(r)\leq r$ for all $r\in\R_+$.

Hence $\min\big\{\|\phi(s,x_k,d)\|_X \ |\  s \in[t_k-\tau,t_k]\big\} \geq \min\{\eps,\sigma^{-1}(\frac{\eps}{2})\}= \sigma^{-1}(\frac{\eps}{2})$ and \eqref{eq:Vintbound} implies for every $k$ 
\begin{equation*}
    V(x_k) \geq \int_{t_k-\tau}^{t_k} \alpha(\|\phi(s,x_k,d)\|_X)ds \geq 
    \alpha \circ \sigma^{-1}\Big(\frac{\varepsilon}{2}\Big) \tau > 0.
\end{equation*}
This contradiction proves uniform stability of $\Sigmafp$.

\textbf{Step 2: (Uniform global attractivity)} Again we assume that $\Sigmafp$ is
not uniformly globally attractive. This implies that there are $r,\varepsilon>0$
and sequences $\{ x_k \}_{k\in\N}$ in $X$, $\{ d_k \}_{k\in\N}$ in $\mathcal{D}$ 
and times $t_k \to \infty$, as $k\to \infty$ such that 
\begin{equation}
\label{ass-noUA}
    \|x_k\|_X \leq r \text{ and } \| \phi(t_k,x_k,d_k) \|_X \geq \varepsilon.
\end{equation}
As we have already shown that $\Sigmafp$ is uniformly stable 
we may choose for the above $\eps$ a $\delta = \delta(\eps)> 0$ such that $\|z\|_X < \delta$ implies 
\begin{equation}
    \label{eq:5}
    \| \phi(t,z,d ) \|_X \leq \frac{\varepsilon}{2} \quad \forall \, t\geq 0, \forall\, d\in \mathcal{D}.
\end{equation}
Now assume that there exist a certain $k \in\N$ and $s_k \in [0,t_k]$ so that $\| \phi(s_k,x_k,d_k) \|_X \leq \delta$.
Since $\Sigma$ satisfies the cocycle property~($\Sigma$\ref{axiom:Cocycle}), \eqref{eq:5} and \eqref{ass-noUA} lead us to 
\begin{align*}
\eps \leq \| \phi(t_k,x_k,d_k) \|_X 
=  \big\| \phi\big(t_k-s_k,\phi(s_k,x_k,d_k),d_k(s+\cdot)\big) \big\|_X \leq \frac{\eps}{2},
\end{align*}
which is a contradiction.
We conclude that for all $k$ and all $t\in [0,t_k]$ we have
\begin{equation}
    \label{eq:7}
  \| \phi(t,x_k,d_k) \|_X \geq \delta.  
\end{equation}
It then follows with \eqref{ass-noUA}, \eqref{eq:1} and
\eqref{eq:Vintbound} that for all $k\geq 1$
\begin{align}
\label{eq:6}
    \psi_2(r) \geq \psi_2(\|x_k\|_X) 
		\geq V(x_k) \geq \int_0^{t_k}\hspace{-3mm} \alpha\left(\| \phi(t,x_k,d_k) \|_X\right)dt
    \geq \alpha(\delta) t_k.
\end{align}
As $t_k\to \infty$, this is a contradiction and hence $\Sigmafp$ is uniformly globally attractive.

Since $\Sigma$ is robustly forward complete and $0$ is a robust equilibrium point of $\Sigma$, Proposition~\ref{prop:UGAS_Characterization} ensures that $\Sigmafp$ is UGAS.
\end{proof}

\begin{theorem}{\textbf{(Non-coercive local UAS Lyapunov  theorem)}}
    \label{t:noncoeLT-local}

    Consider a system $\Sigma=(X,\Dc,\phi)$  and assume that $\Sigma$ is
    robustly forward complete and $0$ is a robust equilibrium of $\Sigma$.
    Let $V$ be a
    non-coercive Lyapunov function in the sense of
    Definition~\ref{def:UGAS_LF_With_Disturbances} and assume that
    $\alpha$ in \eqref{DissipationIneq_UGAS_With_Disturbances} is positive
    definite. Then $\Sigma$
    is locally uniformly asymptotically stable in $x^* =0$.
\end{theorem}

\begin{proof}
    We just point out the necessary modifications in the proof of
    Theorem~\ref{t:noncoeLT}. The derivation of the integral bound
    \eqref{eq:Vintbound} does not depend on a particular property of
    $\alpha$, so that it also holds under the assumptions of the present
    theorem. To obtain stability we argue as in Step 1 of the previous
    proof. 
	We only need to change the final estimate to obtain
\begin{eqnarray*}
    V(x_k) \geq \int_{t_k-\tau}^{t_k} \alpha(\|\phi(s,x,d)\|_X)ds \geq 
    \min \Big\{ \alpha( \xi) \ |\  \xi \in [\sigma^{-1}\Big(\frac{\varepsilon}{2}\Big),\eps]\Big\}    \tau > 0.
\end{eqnarray*}
This again yields the contradiction which proves uniform stability.

Given that we have shown uniform stability, we may choose an
$R>0$ and a corresponding $r>0$ such that all solutions
with initial condition in $B_r$ remain in $B_R$ for all
times.  The proof of locally uniform attractivity uses precisely the same
arguments as before but restricted to solutions with initial conditions in
$B_r$. The only necessary change in the argument is then to replace the
term $\alpha(\delta) t_k$ on the right hand side of \eqref{eq:6} by
\begin{equation*}
    \min \{ \alpha( \xi) \ |\ \xi \in [\delta , R] \} t_k,
\end{equation*}
which uses the invariance argument we have just noted. Again this
expression is unbounded as $k \to \infty$ and we obtain the desired
contradiction. 
\end{proof}

Theorem~\ref{t:noncoeLT} shows that the existence of a non-coercive Lyapunov function implies UGAS in $0$ provided $\Sigma$ is robustly forward complete and $0$ is a robust equilibrium of $\Sigma$.
At the same time the examples in Section~\ref{sec:examples} show that the 
existence of a non-coercive Lyapunov function alone does not even imply
forward completeness. 
And together with forward completeness it still does not imply robust forward completeness and robustness of the zero equilibrium, see \cite[Remark 4]{HaS11} and equation~\eqref{eq:HanteSigalottiEx} in this paper.
Nevertheless, the following holds:
\begin{proposition}
    \label{prop:noncoeLT_plus_FC_implies_ULIM_v2}
    Consider a forward complete system $\Sigma=(X,\Dc,\phi)$. If there
    exists a non-coercive Lyapunov function $V$ for $\Sigma$, then $0$ is
    uniformly globally weakly attractive.
\end{proposition}

\begin{proof}
    Let $\alpha\in \K$ and $\psi_2\in\Kinf$ be the functions
    characterizing the decay condition in \eqref{DissipationIneq_UGAS_With_Disturbances} and
    the upper bound of $V$ in
    \eqref{LyapFunk_1Eig_UGAS}, respectively.  As in
    the proof of Theorem~\ref{t:noncoeLT}, forward completeness and the
    existence of a non-coercive Lyapunov function imply the existence of a
    $\psi_2\in\Kinf$ such that for all $x\in X, d\in\Dc, t\geq 0$
\begin{eqnarray}
\int_0^t \alpha(\|\phi(s,x,d)\|_X)ds \leq V(x) \leq \psi_2(\|x\|_X).
\label{eq:Vintbound_2_v2}
\end{eqnarray}
Fix $\eps>0$ and $r>0$ and define $\tau(\eps,r):=\frac{\psi_2(r) +
  1}{\alpha(\eps)}$. We claim that this choice of $\tau = \tau(\eps,r)$
yields an appropriate time to conclude uniform weak attraticivity of
$0$. Assume to the contrary that there exist $x\in X$ with $\|x\|_X \leq r$ and some $d\in \Dc$ with the property
$\|\phi(t,x,d)\|_X \geq \eps$ for all $t\in[0,\tau(\eps,r)]$. In view of \eqref{eq:Vintbound_2_v2} and since $\alpha\in\K$, we obtain
\[
\psi_2(r) + 1 = \tau(\eps,r) \alpha(\eps) \leq  \int_0^{\tau(\eps,r)} \alpha(\|\phi(s,x,d)\|_X)ds \leq \psi_2(r),
\]
a contradiction.
\end{proof}

\section{Applications}
\label{sec:Examples_I}

In this section we give a few examples of system classes that are covered
by our assumptions. 
Subsumed are, of course, the systems of ordinary
differential equations with uniformly bounded disturbances studied in
\cite{SoW96}.
This
class can be extended to systems of ordinary
differential equations on Banach spaces using the tools described in
\cite{AuC96}. We do not dwell on this and prefer to present examples in
which the unboundedness of generators of $C_0$ semigroups may play a role.

\subsection{Homogeneous systems}

Some of the arguments can be simplified if the system exhibits a homogeneity property.
We call a system $\Sigma=(X,\Dc,\phi)$ homogeneous in $x$ (of order one) if for
all $\lambda \geq 0 $, all initial conditions $x \in X$, $d \in \mathcal{D}$ and
$t\geq 0$ we have 
\begin{equation*}
    \phi( t, \lambda x, d) = \lambda \phi(t,x,d).
\end{equation*}
In particular, systems which are linear in $x$ are homogeneous in $x$. In
the sequel we will simply speak of homogeneous systems. Some results on
Lyapunov functions for homogeneous systems are already available in
\cite[Theorem~14.3]{Dei92}.

For homogeneous systems robustness of the equilibrium point and robust
forward completeness are equivalent.

\begin{lemma}
    \label{lem:hom.REP=RFC}
    Consider a homogeneous system
    $\Sigma=(X,\Dc,\phi)$. The equilibrium point $x^* = 0$ is robust if
    and only if $\Sigma$ is robustly forward complete.
\end{lemma}

\begin{proof}
Let $0$ be a robust equilibrium of $\Sigma$. Then for the choice $\eps=1$ and any $\tau>0$ there is a $\delta>0$ so that
\begin{equation}
    \label{eq:29}
\|x\|_X\leq\delta,\ d\in\Dc,\ t\in[0,\tau] \qrq \|\phi(t,x,d)\|_X\leq 1.    
\end{equation}
Let $C>0$ and $\tau>0$ be arbitrary and choose a $\delta = \delta(\tau)$
such that \eqref{eq:29} is satisfied. Consider $x \in B_C$ and an
arbitrary $d\in\Dc$. If $\|x\|_X\leq \delta$, then the solution is bounded
by $1$ on $[0,\tau]$ and there is nothing to show. Otherwise, let $\lambda
:= \delta/\|x\|_X$. Then we obtain from homogeneity and \eqref{eq:29}
\begin{eqnarray*}
\|\phi(t,x,d)\|_X =  \frac{1}{\lambda} \|\phi(t,\lambda x, d)\|_X 
\leq \frac{\|x\|_X}{\delta} \leq \frac{C}{\delta}. 
\end{eqnarray*}
This shows that $\Sigma$ is RFC.    

Conversely, assume that $\Sigma$ is RFC and fix $\varepsilon >0$ and $\tau>0$. By assumption we have
\begin{equation}
    \label{eq:30}
  \sup\{ \|\phi(t,x,d)\|_X
\ |\  \|x\|_X\leq 1,\: t \in [0,\tau],\: d\in \Dc \} := \Delta < \infty.  
\end{equation}
If $\Delta \leq \varepsilon$ we are done. Otherwise, let $\delta :=
\varepsilon/ \Delta$. If $x\in X$ with $\|x\|_X\leq \delta$ and $t \in
[0,\tau]$ it follows for all $d\in \mathcal{D}$ that
\begin{equation*}
   \|\phi(t,x,d)\|_X = \delta
   \left\|\phi\left(t,\frac{1}{\delta}x,d\right)\right\|_X \leq
   \frac{\varepsilon}{\Delta} \Delta = \varepsilon.
\end{equation*}
This shows robustness of the equilibrium point.
\end{proof}

\begin{corollary}
    \label{c:homogeneous}
    Consider a homogeneous system $\Sigma=(X,\Dc,\phi)$ and assume that
    $V$ is a non-coercive Lyapunov function in the sense of
    Definition~\ref{def:UGAS_LF_With_Disturbances} and assume that
    $\alpha$ in \eqref{DissipationIneq_UGAS_With_Disturbances} is positive
    definite.  If 
    $0$ is a robust equilibrium of
    $\Sigma$, then $\Sigma$ is uniformly globally asymptotically stable in $x^* =0$.
\end{corollary}

\begin{proof}
    By Lemma~\ref{lem:hom.REP=RFC} $\Sigma$ is RFC.  We then have local uniform
    asymptotic stability by Theorem~\ref{t:noncoeLT}. Global uniform
    attractivity then follows by homogeneity using a simple scaling
     argument similar to those in the proof of
      Lemma~\ref{lem:hom.REP=RFC}.  
\end{proof}

\subsection{Switched linear systems in Banach spaces}

This class of infinite-dimen\-sional switched linear systems has been
studied in \cite{HaS11}. Further results for the special case of switched
linear delay systems are obtained in \cite{haidar2015converse}. Here we briefly
outline how to recover some of the results of \cite{HaS11} with the
arguments proposed here.  Let $X$ be a Banach space. Consider a set of
closed linear operators $\{ A_q \ |\  q\in Q \}$ where $Q$ is some index
set. We assume that each $A_q$ generates a $C_0$-semigroup $T_q$ on
$X$. Let
\begin{equation}
    \label{eq:9}
    \Dc := \{ d : \R_+ \to Q \ |\  d \text{ is piece-wise constant}\},
\end{equation}
where piece-wise constant means here that any half-open bounded interval $[a,b)
\subset \R_+$ can be partitioned into finitely many half open intervals
$[a_i,b_i)$ such that  $d$ is constant on each $[a_i,b_i)$. (The choice of
using right-closed intervals is mere convention and nothing would change,
if we were to use left-closed intervals instead.) We consider the
family of differential equations
\begin{equation}
    \label{eq:10}
    \dot x = A_{d(t)} x(t)
\end{equation}
which generates evolution operators in the following manner. For $d\in
\Dc$ and an interval $[s,t]$ with a partition $s=b_0 < b_1 < \ldots < b_k
= t$ such that $d \equiv d_j \in Q$ on $[b_{j-1},b_j)$, $j=1,\ldots,k$ we set
\begin{equation}
    \label{eq:11}
    \Phi_d(t,s) = T_{d_k}(t-b_{k-1}) T_{d_{k-1}}(b_{k-1} - b_{k-2}) \dots T_{d_1}(b_1-s).
\end{equation}
With this notation we have $\phi(t,x,d) = \Phi_d(t,0)x$ for all $x\in X,
d\in \Dc, t\geq 0$. It is easy to check that the conditions of
Definition~\ref{Steurungssystem} are all satisfied. We also have the
following lemma.

\begin{lemma}
    \label{lemma-switchRFC} Consider the system $\Sigma = (X,\Dc,\phi)$
    given by switched linear system \eqref{eq:10} with $\Dc$ as defined in
    \eqref{eq:9}. The following statements are equivalent.
    \begin{enumerate}[(i)]
      \item $\Sigma$ is robustly forward complete.
      \item $x^\ast = 0$ is a robust equilibrium point of $\Sigma$.
      \item There exist constants $M\geq 1$, $\omega\in \R$ such that
        \begin{equation}
            \label{eq:12}
            \| \Phi_d(t,s)  \| \leq M e^{\omega (t-s)}\quad \forall \ d\in
            \Dc,\ t\geq s \geq 0.
        \end{equation}
    \end{enumerate}
\end{lemma}

\begin{proof}
The equivalence (i) $\Leftrightarrow$ (ii) is an immediate consequence of
Lemma~\ref{lem:hom.REP=RFC}. 

    (iii) $\Rightarrow$ (i). If (iii) holds then we can define
    $\mu(C,\tau) := C M e^{\omega \tau}$ to see that for all solutions we
    have $\|\phi(t,x,d)\|_X \leq \mu(\|x\|_X,t)$. Thus $\Sigma$ is RFC by
    Lemma~\ref{lem:RFC_criterion}.
   
   (ii) $\Rightarrow$ (iii). Fix $h>0$. Using
   Lemma~\ref{lem:0-REP_criterion} there exists a $\delta>0$ and a
   function $\tilde\mu(\cdot,h): [0,\delta) \to \R_+$ such that for all $x\in X$
   with $\|x\|_X = \delta/2$ we have
   \begin{equation}
       \label{eq:15}
      \| \Phi_d(t,0)x\|_{\mathcal{L}(X)} \leq \tilde\mu\left( \frac{\delta}{2},h\right) \quad
      \forall d\in\Dc, t\in [0,h].  
   \end{equation}
   By linearity we thus obtain for all $x\in X$ a bound of $\|
   \Phi_d(t,0)x\|$ which is uniform in $d\in \Dc$. An application of the Banach-Steinhaus theorem yields that
   there is a constant $\tilde M > 0$ such that 
   \begin{equation}
       \label{eq:16}
       \sup \left \{ \| \Phi_d(t,0) \|_{\mathcal{L}(X)} \ |\  t\in [0,h],
           d\in \Dc \right \} \leq \tilde M.
   \end{equation}
   Take an arbitrary $d \in \Dc$ and $t\geq 0$. Write $ t = k h + \tau$,
   where $\tau \in [0,h]$. For each $j= 1,\ldots,k$ the shift $d(\cdot +
   jh) \in \Dc$ so that we obtain
   \begin{equation}
       \label{eq:17}
      \| \Phi_d(t,0) \|_{\mathcal{L}(X)} \leq \|
      \Phi_{d(\cdot+kh)}(\tau,0)\|_{\mathcal{L}(X)} \prod_{j=0}^{k-1} \|
        \Phi_{d(\cdot+jh)}(h,0)\|_{\mathcal{L}(X)} \leq \tilde{M}^{k+1}.
   \end{equation}
   From the last inequality it is easy to arrive at the desired \eqref{eq:12}.
\end{proof}

\begin{remark}
    An immediate consequence of the characterization (iii) of
    the previous lemma is that for linear switched systems the 
    flow of $\Sigma$ is Lipschitz
    continuous if and only if the system is robustly forward complete. 
\end{remark}

For switched linear systems on Banach space we thus obtain the following
results
\begin{corollary}
    \label{c:scwitched_cor}
    Consider a switched linear system $\Sigma=(X,\Dc,\phi)$ as described by
    \eqref{eq:9}-\eqref{eq:11}. 
\begin{enumerate}[(i)]
  \item If there exists a non-coercive Lyapunov function $V$ for $\Sigma$
    then $0$ is uniformly globally weakly attractive.
  \item If $\Sigma$ is RFC, then the following two statements are
    equivalent
    \begin{enumerate}[(a)]
      \item there exists a non-coercive Lyapunov function $V$ for
        $\Sigma$.
      \item $0$ is uniformly globally asymptotically stable and hence
        exponentially stable.
    \end{enumerate}
\end{enumerate}
\end{corollary}

\begin{proof}
    It is clear that switched linear systems as described by
    \eqref{eq:9}-\eqref{eq:11} are forward complete. Then
    Proposition~\ref{prop:noncoeLT_plus_FC_implies_ULIM_v2} shows (i). The
    sufficiency part ``(a) $\Rightarrow$ (b)'' in (ii) follows from
    Theorem~\ref{t:noncoeLT} and Lemma~\ref{lemma-switchRFC}. Necessity is
    a consequence of Theorem~\ref{LipschitzConverseLyapunovTheorem-2}. 
\end{proof}

While item (i) of the previous result is new (to the best knowledge of the
authors), item (ii) recovers some of the results of
Theorem 3 in \cite{HaS11}.
We note that Remark 4 in \cite{HaS11} also shows that even for this system
class the assumption of robust forward completeness cannot be removed in
order to conclude uniform global asymptotic stability. This will be
discussed in the examples of Section~\ref{sec:examples}. A further
property is that the Lyapunov function may in fact be chosen to be a norm
on $X$ (or the square of a norm, which really makes no difference). In the
non-coercive case, this norm is of course not equivalent to the original
norm.

\subsection{Strongly continuous semigroups with admissible input operators}

Let $X$ be a Banach space and $A$ be the generator of a $C_0$-semigroup $T(\cdot)$ on
$X$. Consider the extrapolation space $X_{-1}$ defined as the completion
of $X$ with respect to the norm
\begin{equation}
    \label{eq:18}
    \|x \|_{-1} := \| (s I - A)^{-1} x\|_X\,,
\end{equation}
where $s$ is an arbitrary element of the resolvent set of $A$. In this
case $T(\cdot)$ extends to a $C_0$-semigroup on $X_{-1}$ which we denote by
the same symbol.
Consider a
linear operator $B \in \mathcal{L}(U,X_{-1})$, which may therefore be
unbounded as an operator from $U$ to $X$. Given the space of input
functions $L^p(\R_+,U)$, the operator $B$ is called $p$-admissible if for
all $t>0$ the
map
\begin{equation}
    \label{eq:19}
  G_t : u(\cdot) \mapsto \int_0^t T(t-s) u(s) ds   
\end{equation}
defines a bounded operator from $L^p(\R_+,U)$ to $X$. Note that the
integral is a priori defined in $X_{-1}$ and it is a requirement that the
integral yields a value in $X$. In the particular case that $p=\infty$ we
require the further condition that $\| G_t
\|_{\mathcal{L}(L^\infty(\R_+,U),X)} \to 0$ as $t \to + 0$. If this is the
case then $B$ is called zero-class admissible \cite{JPP09}. 

Consider a linear system 
\begin{eqnarray}
\dot x = Ax + Bu,
\label{eq:LinSys}
\end{eqnarray}
where $A$ is the generator of a $C_0$-semigroup over the Banach space $X$,
and $B$ is admissible for $A$.  It is known that if $B$ is $p$-admissible
for $1\leq p < \infty$ or $\infty$-admissible and zero-class, then the
solutions $\phi(\cdot,x,d)$ are continuous, due to an argument similar to
\cite[Proposition 2.3]{Wei89b}.

This system class gives rise to a system with disturbances if we consider
uncertain time-varying feedbacks. In this way we can study unbounded
perturbations of the generator $A$, see e.g. \cite{gallestey2000spectral}. To this
end consider a system of the form
\begin{equation}
    \label{eq:31}
    \dot x = Ax + B\Delta(t)Cx, 
\end{equation}
where $(A,B)$ is (for instance) $\infty$-admissible and zero-class, $C \in
\mathcal{L}(X,Y)$, $D\subset \mathcal{L}(Y,U)$ and the set of disturbances
satisfies
\begin{equation*}
    \mathcal{D} = \left \{ \Delta : [0,\infty) \to D \ |\  \Delta \text {
        is piecewise continuous } \right\}.   
\end{equation*}
We consider mild solutions to \eqref{eq:31}, i.e. solutions to the
integral equation
\begin{equation*}
    \phi(t,x_0,\Delta) = T(t) x_0 + \int_0^t T(t-s) B\Delta(s) C x(s) ds.
\end{equation*}
Conditions for the existence of such solutions are discussed in
\cite{hinrichsen1994robust,jacob1995infinite,chen2015time}. Here we will
assume that \eqref{eq:31} defines a system $\Sigma=(X,\Dc,\phi)$.

System \eqref{eq:31} is homogeneous, so that we obtain that provided the
equilibrium $x^*=0$ is robust, the existence of a non-coercive Lyapunov
function guarantees uniform global asymptotic stability of $0$.

\subsection{Nonlinear evolution equations in Banach spaces}

We consider nonlinear infinite-dimensional systems of the form
\begin{equation}
\label{InfiniteDim}
\dot{x}(t)=Ax(t)+g(x(t),d(t)), \ x(t) \in X,\ d(t) \in D,
\end{equation}
where $A$ generates the $C_0$-semigroup $T(\cdot)$ of bounded operators,
$X$ is a Banach space and $D$ is a nonempty subset of a normed linear
space.  As the space of admissible inputs we consider the
space ${\Dc}$ of globally bounded, piecewise continuous functions from
$\R_+$ to $D$.

\begin{Ass}
\label{Assumption1}
The function $g:X \times D \to X$ is Lipschitz continuous on bounded subsets of $X$, uniformly with respect to the second argument, i.e.
for all $C>0$ there exists $L_f(C)>0$, such that for all $x,y$ with
$\|x\|_X \leq C,\ \|y\|_X \leq C$ and $d \in D$ it holds that
\begin{eqnarray}
\|g(y,d)-g(x,d)\|_X \leq L_f(C) \|y-x\|_X.
\label{eq:Lipschitz}
\end{eqnarray}
Assume also that $g(x,\cdot)$ is continuous for all $x \in X$.
\end{Ass}

We consider mild solutions of \eqref{InfiniteDim}, i.e. solutions of the integral equation
\begin{equation}
\label{InfiniteDim_Integral_Form}
x(t)=T(t)x(0) + \int_0^t T(t-s)g(x(s),d(s))ds 
\end{equation}
belonging to the class $C([0,\tau],X)$ for certain $\tau>0$. 

It is well known that the system \eqref{InfiniteDim} is well-posed if Assumption~\ref{Assumption1} is satisfied. Moreover it satisfies all the axioms of the 
Definition~\ref{Steurungssystem}, possibly with exception of
forward completeness. Thus, \eqref{InfiniteDim} gives rise to a control system $\Sigma=(X,\Dc,\phi)$.

We are going to show that for system \eqref{InfiniteDim} some of the assumptions of Theorem~\ref{t:noncoeLT} are automatically satisfied. 
\begin{lemma}
\label{lem:Regularity}
Assume that
\begin{enumerate}
	\item[(i)] \eqref{InfiniteDim} is robustly forward complete,
	\item[(ii)] Assumption~\ref{Assumption1} holds.
\end{enumerate}
Then \eqref{InfiniteDim} has a flow which is Lipschitz continuous on compact intervals.
\end{lemma}

\begin{proof}
Let $C>0$, $d \in \Dc$, and $x^0_1,x^0_2$ with  $\|x^0_i\|_X \leq C$, $i=1,2$. 
Denote $x_i(t):=\phi(t,x^0_i,d)$, $i=1,2$ be the solutions of
\eqref{InfiniteDim}  which are defined for $t\geq0$ since we assume that
\eqref{InfiniteDim} is forward complete.

There exist $M>0,\lambda \in\R$ such that  $\|T(t)\| \leq Me^{\lambda t}$
for all $t \geq 0$. Pick any $\tau >0$ and set 
\[
K(C,\tau):= \sup_{\|x\|_X\leq C, d\in \Dc, t \in [0,\tau]}\|\phi(t,x,d)\|_X,
\] 
which is finite due to assumption (i). For any $t \in [0,\tau]$ we have that
\begin{align*}
\|x_1(t) - x_2(t)\|_X &\leq \|T(t)\| \|x^0_1-x^0_2\|_X  \\
&\qquad\qquad + \int_0^t{\|T(t-r)\| \|g(x_1(r),d(r))-g(x_2(r),d(r))\|_Xdr} \\
&\leq Me^{\lambda t} \|x^0_1-x^0_2\|_X + \int_0^t Me^{\lambda (t-r)}L(K(C,\tau)) \|x_1(r)-x_2(r)\|_X dr.
\end{align*}
Now define $z_i(t):=e^{-\lambda t}x_i(t)$, $i=1,2$, $t \geq 0$.
Then
\begin{align*}
\|z_1(t) - z_2(t)\|_X \leq M \|x^0_1-x^0_2\|_X + M L(K(C,\tau)) \int_0^t{ \|z_1(r)-z_2(r)\|_X dr}. 
\end{align*}
According to Gronwall's inequality we obtain
\begin{align*}
\|z_1(t) - z_2(t)\|_X  \leq M \|x^0_1-x^0_2\|_X e^{ M L(K(C,\tau)) t} 
\end{align*}
or in the original variables, for $t\in[0,\tau]$
\begin{align*}
\|\phi(t,x^0_1,d) - \phi(t,x^0_2,d)\|_X  \leq  \|x^0_1-x^0_2\|_X e^{ (M L(K(C,\tau)) + \lambda) \tau}. 
\end{align*}
which proves the lemma. 
\end{proof}

Summarizing the results of Theorem~\ref{t:noncoeLT}, Lemma~\ref{lem:Regularity} and Lemma~\ref{lem:RobustEquilibriumPoint},
we obtain:
\begin{corollary}
Let Assumption~\ref{Assumption1} be satisfied. Let \eqref{InfiniteDim} be robustly forward complete and let $0$ be an equilibrium of \eqref{InfiniteDim}. If there exists a non-coercive Lyapunov function for \eqref{InfiniteDim}, then \eqref{InfiniteDim} is UGAS.
\end{corollary}

\begin{proof}
Lemma~\ref{lem:Regularity} and RFC property of \eqref{InfiniteDim} imply that the flow of \eqref{InfiniteDim} is Lipschitz continuous on compact intervals. Next Lemma~\ref{lem:RobustEquilibriumPoint} implies that $0$ is a robust equilibrium point of \eqref{InfiniteDim}. Finally, Theorem~\ref{t:noncoeLT} shows that \eqref{InfiniteDim} is UGAS.
\end{proof}

\section{Converse theorems}
\label{sec:converse-theorems}

In this section we consider two constructions of Lyapunov functions for
UGAS systems $\Sigma=(X,\Dc,\phi)$. First we recall a construction of
Lipschitz continuous coercive Lyapunov functions, due to \cite{KaJ11b},
and then we use these ideas in order to derive another construction of a
non-coercive Lyapunov function for a system with a UGAS
equilibrium.
The construction requires Lipschitz continuity of the flow -  a
condition we have hardly used so far. For a detailed discussion of the
properties required of the set of solutions of a system to obtain a
converse Lyapunov theorem yielding continuous Lyapunov functions we refer to \cite{Sch15}.

Both of the constructions exploit Sontag's $\KL$-lemma \cite[Proposition
7]{Son98}, which can be stated as follows:
\begin{lemma}
\label{Sontags_KL_Lemma}
For each $\beta \in \KL$ there exist $\alpha_1,\alpha_2 \in \Kinf$ such
that 
\begin{equation}
\beta(r,t) \leq \alpha_2(\alpha_1(r)e^{-t}) \quad \forall r \geq 0, \; \forall t \geq 0.
\label{eq:KL-Lemma_Estimate}
\end{equation}
\end{lemma}

We will need several technical lemmas, which are well-known. 
\begin{lemma}
\label{lem:SimpleIneqLemma}
Let $f,g:\Dc \to \R_+$ be maps such that  $\sup_{d \in \Dc}f(d)$ and $\sup_{d \in \Dc}g(d)$ are finite.
Then
\begin{equation}
\sup_{d \in \Dc}f(d) - \sup_{d \in \Dc}g(d) \leq \sup_{d \in \Dc}\big(f(d) - g(d)\big).
\label{eq:SimpleIneqLemma}
\end{equation}
\end{lemma}

We note that Lemma~\ref{lem:SimpleIneqLemma} does not claim that the
expression on the right in \eqref{eq:SimpleIneqLemma} is finite.
The following lemma is taken from \cite[p.130]{KaJ11b}.
\begin{lemma}
\label{lem:KinfLipschitzLowerEstimate}
For any $\alpha \in \Kinf$ there exist $\rho \in \Kinf$ so that $\rho(s) \leq \alpha(s)$ for all $s \in \R_+$ and $\rho$ is globally Lipschitz with a unit Lipschitz constant, i.e. for any $s_1,s_2 \geq 0$ it holds that 
\begin{equation}
|\rho(s_1) - \rho(s_2)| \leq |s_1 - s_2|.
\label{eq:Lipschitz_Kinf}
\end{equation}
\end{lemma}

We will also need the functions $G_k:\R_+ \to \R_+$ defined,
for $k \in \N$, by
\[
G_k(r):=\max\{r-\frac{1}{k},0 \}, \quad r \geq 0.
\]
\begin{lemma}
\label{Converse_Lyapunov_Theorem_SimpleEstimate}
For each $k \in \N$ the function $G_k$ is Lipschitz continuous with a unit Lipschitz constant, i.e., for all $r_1,r_2 \geq 0$ it holds that 
\begin{equation}
\big|G_k(r_1) - G_k(r_2)\big| \leq |r_1-r_2|.
\label{eq:CLT_Simp_Est}
\end{equation}
\end{lemma}

\begin{proof}
This follows as each $G_k$ is the maximum of two Lipschitz continuous
functions with Lipschitz constant (at most) $1$.
\end{proof}

The following theorem has been shown in \cite[Section 3.4]{KaJ11b}.
\begin{theorem}
\label{LipschitzConverseLyapunovTheorem-1}
Consider a system $\Sigma= (X,\Dc,\phi)$  with a flow, which is Lipschitz continuous on compact intervals.
If $\Sigmafp$ is UGAS, then for any $\eta \in (0,1)$ there exists a sequence
of positive real numbers $\{ a_k \}_{k\in\N}$ so that $W^\eta:X \to \R_+$, defined by
\begin{eqnarray}
W^\eta (x) := \sum_{k=1}^{\infty} a_k V^\eta_k (x) \qquad \forall x \in X,
\label{eq:LF_ConverseLyapTheorem_maxType_W-LF}
\end{eqnarray}
with 
\begin{eqnarray}
V^\eta_k (x) := \sup_{d \in \Dc} \max_{s \in [0,+\infty)}e^{\eta s} G_k\big(\rho( \| \phi(s,x,d)\|_X)\big).
\label{eq:LF_ConverseLyapTheorem_maxType_modified_Def2}
\end{eqnarray}
is a coercive Lyapunov function for $\Sigma$ which is Lipschitz continuous on bounded balls.
\end{theorem}
The proof of this result is based on the earlier local converse Lyapunov theorems, see e.g. \cite[Theorem 19.3]{Yos66}, \cite[Theorem 4.2.1]{Hen81} and on using Sontag's $\KL$-Lemma (Lemma~\ref{Sontags_KL_Lemma}).

Next, exploiting ideas from \cite[Section 3.4]{KaJ11b} we construct a Lipschitz continuous non-coercive Lyapunov function using integration instead of maximization.
\begin{theorem}
\label{LipschitzConverseLyapunovTheorem-2}
Consider a forward complete 
system $\Sigma= (X,\Dc,\phi)$ with a flow, which is Lipschitz continuous on compact intervals.
If $\Sigmafp$ is UGAS, then there exists a non-coercive Lyapunov function for $\Sigma$ which is Lipschitz continuous on bounded balls.
\end{theorem}

\begin{proof}
Let the system $\Sigma=(X,\Dc,\phi)$  be UGAS in $x^*=0$. Then there exists a $\beta \in \KL$, so that \eqref{UGAS_wrt_D_estimate} holds for all $x \in X$, all $d \in \Dc$ and all $t \geq 0$.
Due to Lemma~\ref{Sontags_KL_Lemma} there exist $\alpha_1,\alpha_2 \in \Kinf$ so that \eqref{eq:KL-Lemma_Estimate} holds.

Pick any Lipschitz continuous function $\rho \in \Kinf$ with a unit
Lipschitz constant so that $\rho(r) \leq \alpha_2^{-1}(r)$, $r \in
\R_+$. This is possible due to
Lemma~\ref{lem:KinfLipschitzLowerEstimate}. For any $k \in \N$ define
$G_k(r):=\max\{r-\frac{1}{k},0 \}$ and consider the function
\begin{eqnarray}
V_k(x)= \sup_{d \in \Dc} \int_0^{\infty} G_k\big(\rho\big(\|\phi(t,x,d)\|_X\big) \big) dt.
\label{eq:LF_ConverseLyapTheorem_With_Disturbances_Gk}
\end{eqnarray}

For all $x \in X$ it holds that
\begin{eqnarray*}
V_k(x)&\leq& \sup_{d \in \Dc} \int_0^{\infty} \rho\big(\|\phi(t,x,d)\|_X\big) dt \\
&\leq & \int_0^{\infty} \alpha_2^{-1}\big(\beta(\|x\|_X,t)\big) dt \\
&\leq & \int_0^{\infty} \alpha_1(\|x\|_X)e^{-t}  dt \\
& = & \alpha_1(\|x\|_X).
\end{eqnarray*}

Clearly, $V_k(x) \geq 0$ for all $x \in X$. Moreover, if $\|x\|_X >
\rho^{-1}(\tfrac{1}{k})$ holds, then, due to the continuity axiom~($\Sigma$\ref{axiom:Continuity}) for any given $d \in \Dc$ there exists $t_0 = t_0(d) >0$ so that 
$\|\phi(t,x,d)\|_X > \rho^{-1}(\tfrac{1}{k})$  for all $t \in
[0,t_0)$. Thus $G_k\big(\rho\big(\|\phi(t,x,d)\|_X\big) \big) > 0 $ for
all $t \in [0,t_0)$, which implies that $V_k(x)>0$ provided that $\|x\|_X > \rho^{-1}(\tfrac{1}{k})$.

The Dini derivative of $V_k$ at any $x \in X$ for $v \in \Dc$ is given by
\begin{eqnarray*}
\dot{V_k}_v(x)&=&\mathop{\overline{\lim}} \limits_{h \rightarrow +0} {\frac{1}{h}\Big(V_k\big(\phi(h,x,v)\big)-V_k(x)\Big) } \\
 &=&\mathop{\overline{\lim}} \limits_{h \rightarrow +0} \frac{1}{h}\Big(
\sup_{d \in \Dc} \int_0^{\infty} G_k\big(\rho(\|\phi(t,\phi(h,x,v),d)\|_X)\big) dt \\
&&\qquad\qquad - \sup_{d \in \Dc} \int_0^{\infty} G_k\big(\rho(\|\phi(t,x,d)\|_X) \big) dt\Big)\\
 &=&\mathop{\overline{\lim}} \limits_{h \rightarrow +0} \frac{1}{h}\Big(
\sup_{d \in \Dc} \int_0^{\infty} G_k\big(\rho(\|\phi(t+h,x,\tilde{d})\|_X)\big) dt \\
&&\qquad\qquad - \sup_{d \in \Dc}  \int_0^{\infty} G_k\big(\rho(\|\phi(t,x,d)\|_X)\big) dt\Big),
\end{eqnarray*}
where the disturbance function $\tilde{d}$ is defined as 
\[
\tilde{d}(t) := 
\begin{cases}
v(t), & \text{ if } t \in [0,h] \\ 
d(t-h)  & \text{ otherwise}.
\end{cases}
\]
The supremum in the first integral is taken over disturbances of a
particular form. The supremum cannot decrease, if we allow general disturbances and hence

\begin{eqnarray*}
\dot{V_k}_v(x)&\leq &\mathop{\overline{\lim}} \limits_{h \rightarrow +0} \frac{1}{h}\Big(
\sup_{d \in \Dc} \int_0^{\infty} G_k\big(\rho(\|\phi(t+h,x,d)\|_X)\big) dt \\
&&\qquad\qquad - \sup_{d \in \Dc}  \int_0^{\infty} G_k\big(\rho(\|\phi(t,x,d)\|_X)\big) dt\Big) \\
			&\leq &\mathop{\overline{\lim}} \limits_{h \rightarrow +0} \frac{1}{h}\Big(
\sup_{d \in \Dc} \int_h^{\infty} G_k\big(\rho(\|\phi(t,x,d)\|_X)\big) dt \\
&&\qquad\qquad - \sup_{d \in \Dc}  \int_0^{\infty} G_k\big(\rho(\|\phi(t,x,d)\|_X)\big) dt\Big). \\
\end{eqnarray*}
Using the inequality \eqref{eq:SimpleIneqLemma} we obtain for all $v \in \Dc$ that
\begin{eqnarray*}
\dot{V_k}_v(x)&\leq &\mathop{\overline{\lim}} \limits_{h \rightarrow +0} \frac{1}{h}
\sup_{d \in \Dc} \Big(\int_h^{\infty} G_k\big(\rho(\|\phi(t,x,d)\|_X)\big) dt \\
&&\qquad\qquad - \int_0^{\infty} G_k\big(\rho(\|\phi(t,x,d)\|_X)\big) dt\Big) \\
			&= & - \mathop{\overline{\lim}} \limits_{h \rightarrow +0} \frac{1}{h}
\sup_{d \in \Dc} \int_0^h G_k\big(\rho(\|\phi(t,x,d)\|_X)\big) dt \\
			&\leq & - \mathop{\overline{\lim}} \limits_{h \rightarrow +0} 
						\frac{1}{h} \int_0^h G_k\big(\rho(\|\phi(t,x,0)\|_X)\big) dt \\
			&=& - G_k\big(\rho(\|x\|_X)\big),
\end{eqnarray*}
where the final equality holds due to axiom~($\Sigma$\ref{axiom:Continuity}) and by continuity of $G_k$ and $\rho$.

Now we show that the $V_k$ are Lipschitz continuous on bounded balls.
Since $\Sigmafp$ is UGAS, it holds for any $d \in \Dc$, $x \in X$ and $t \geq 0$ that
\begin{equation}
\rho\big( \| \phi(t,x,d)\|_X) \leq \alpha_2^{-1}\big(\beta(\|x\|_X,t)\big) \leq e^{-t} \alpha_1(\|x\|_X).
\end{equation}
Pick any $R>0$ and $x \in X$ with $\|x\|_X \leq R$. Then for any $d \in \Dc$ and
for any $t \geq T(R,k) := \ln(1 + k \alpha_1(R))$ it holds that
\begin{equation*}
\rho( \| \phi(t,x,d)\|_X) \leq \frac{1}{k}.    
\end{equation*}
Thus, the domain of integration in the definition of $V_k$ has a finite length,
i.e. for $R >0$ and all $x \in X$ with $\|x\|_X \leq R$ the function $V_k$ can be equivalently defined by
\begin{eqnarray*}
V_k(x)&=& \sup_{d \in \Dc} \int_0^{T(R,k)} G_k\big( \rho(\|\phi(t,x,d)\|_X) \big)dt.
\end{eqnarray*}

Now pick any $x,y \in X$ such that  $\|x\|_X \leq R$  and $\|y\|_X \leq R$ and consider
\begin{eqnarray*}
|V_k(x)-V_k(y)|  &=& \Big|\sup_{d \in \Dc} \int_0^{T(R,k)} G_k\big( \rho\big(\|\phi(t,x,d)\|_X\big) \big)dt \\
&  & \quad   -\sup_{d \in \Dc} \int_0^{T(R,k)} G_k\big( \rho\big(\|\phi(t,y,d)\|_X\big) \big)dt \Big|   \\
&\leq& \sup_{d \in \Dc} \Big|\int_0^{T(R,k)} G_k\big( \rho\big(\|\phi(t,x,d)\|_X\big) \big)dt \\
&  & \qquad   - \int_0^{T(R,k)} G_k\big( \rho\big(\|\phi(t,y,d)\|_X\big) \big)dt \Big|   \\
&\leq& \sup_{d \in \Dc} \int_0^{T(R,k)} \Big| G_k\big( \rho\big(\|\phi(t,x,d)\|_X\big) \big) \\
&  & \qquad\qquad   -  G_k\big( \rho\big(\|\phi(t,y,d)\|_X\big) \big) \Big|dt \\
&\leq& \sup_{d \in \Dc} \int_0^{T(R,k)} \|\phi(t,x,d) - \phi(t,y,d)\|_X dt.
\end{eqnarray*}
Exploiting the Lipschitz continuity of the flow on compact intervals we obtain
\begin{eqnarray*}
|V_k(x)-V_k(y)|  &\leq& \sup_{d \in \Dc} \int_0^{T(R,k)} L(R,k)\|x-y\|_X dt \\
&=& T(R,k) L(R,k)\|x-y\|_X,
\end{eqnarray*}
where $L(R,k)$ is strictly increasing in both arguments.
Define $M(R,k):=T(R,k)L(R,k)$, which is also strictly increasing in both arguments.

Define $W: X\to \R_+$ by
\begin{eqnarray}
W (x) := \sum_{k=1}^{\infty} \frac{2^{-k}}{1+M(k,k)} V_k (x) \qquad \forall x \in X.
\label{eq:LF_ConverseLyapTheorem_integralType_W-LF}
\end{eqnarray}
We have $W (x) \leq \alpha_1(\|x\|_X)$ for all $x \in X$.  Since
$V_k(x)>0$ if $\|x\|_X > \rho^{-1}(\tfrac{1}{k})$ we obtain that $W(x)=0$  $\Iff$ $x = 0$.

Differentiating $W$ along the trajectory, we obtain:
\begin{eqnarray*}
\dot{W} (x) &=& \sum_{k=1}^{\infty} \frac{2^{-k}}{1+M(k,k)} \dot{V}_k (x) 
 \leq  - \sum_{k=1}^{\infty} \frac{2^{-k}}{1+M(k,k)} G_k(\rho(\|x\|_X)) 
\leq  - \psi_1(\|x\|_X),
\label{eq:WdotEstimate}
\end{eqnarray*}
where $\psi_1 \in\Kinf$ is defined for any $r \geq 0$ as
\begin{eqnarray}
\psi_1(r):=\sum_{k=1}^{\infty} \frac{2^{-k}}{1+M(k,k)} G_k\big(\rho(r)\big).
\label{eq:Psi1_Def}
\end{eqnarray}

Now pick any $R >0$ and $x,y \in X$ such that  $\|x\|_X \leq R$ and $\|y\|_X \leq R$. Then it holds that 
\begin{align*}
\big|W^\eta(x) - W^\eta (y)\big| &= \Big| \sum_{k=1}^{\infty} \frac{2^{-k}}{1+M(k,k)} \big(V^\eta_k (x) - V^\eta_k (y)\big)\Big| \\
&\leq  \sum_{k=1}^{\infty} \frac{2^{-k}M(R,k)}{1+M(k,k)} \|x-y\|_X\\
&\leq  \Big(1+ \sum_{k=1}^{[R] + 1} \frac{2^{-k}M(R,k)}{1+M(k,k)}\Big) \|x-y\|_X.
\end{align*}
This shows that $W^\eta$ is a Lyapunov function for $\Sigma$ which is Lipschitz continuous on bounded balls.
\end{proof}

To conclude the developments of this paper, we state the following
characterization of uniform global asymptotic stability.
\begin{corollary}
\label{cor:Final_Corollary}
    Consider a system $\Sigma=(X,\Dc,\phi)$ with a Lipschitz continuous
    flow.  Assume $\Sigma$ is RFC and $0$ is a robust equilibrium of $\Sigma$. 
		Then the following  statements are equivalent:
\begin{itemize}
	\item[(i)] $\Sigmafp$ is UGAS.
	\item[(ii)] $\Sigmafp$ is UGATT.
	\item[(iii)] there exists a coercive Lyapunov function for $\Sigma$.
	\item[(iv)] there exists a non-coercive Lyapunov function for $\Sigma$.
\end{itemize}
\end{corollary}

\section{Examples}
\label{sec:examples}

In this section we discuss the limits of the results which can be obtained
using non-coercive Lyapunov functions.
Theorem~\ref{t:noncoeLT} names three conditions for the conclusion of
uniform global asymptotic stability: (i) the fixed point should be robust,
(ii) the system should be robustly forward complete, (iii) there exists a
non-coercive Lyapunov function. Contrary to the coercive case where
the existence of a Lyapunov function implies REP and RFC, the existence of
a non-coercive Lyapunov function does not imply either REP or RFC. This is
a consequence of \cite[Remark 4]{HaS11}\footnote{We note that a few
  modifications are necessary in \cite[Remark 4]{HaS11} in order to obtain
  the desired example. Most importantly, the
  semigroups $T_j(t)$ on $X=L^p(0,1)$, $j\in\N$ should be defined by
  setting for $f\in X$
\begin{equation}
(T_j(t)f)(s) := 
\begin{cases}
2^{\frac{1}{p}}f(s+t), & \text{ if } s \in [0,1-t] \cap [4^{-j} - t,4^{-j}), \\ 
f(s+t), & \text{ if } s \in [0,1-t] \backslash [4^{-j} - t,4^{-j}), \\ 
0, & \text{ if } s \in (1-t,1]  \cap [0,1].
\end{cases}
\label{eq:HanteSigalottiEx}
\end{equation}}.

\begin{example}
\label{example:notRFC}
    In this example we show that even for finite-dimensional undisturbed systems
    the following properties are possible: (i) the system has a unique
    fixed point, (ii) there exist non-coercive Lyapunov functions which
    satisfy the decay condition~\eqref{DissipationIneq_UGAS_With_Disturbances}, (iii)
    the fixed point is not globally asymptotically stable.

    Since non-coercive Lyapunov functions for a finite-dimensional system are coercive in the neighborhood of zero, existence of such a Lyapunov function implies that the fixed point is locally asymptotically stable. Also it is
    impossible that solutions which are not attracted to the fixed point
    exist for all positive times, as then Theorem~\ref{t:noncoeLT} would
    imply global asymptotic stability (for undisturbed ODE systems forward completeness and robust forward completeness are equivalent, see e.g. \cite[Proposition 5.1]{LSW96}).

    Consider a strictly increasing, bounded, Lipschitz continuous function
    $\delta : \R\to\R$. Assume furthermore
    that $\delta(x) x > 0$ for $x\neq 0$, and for the sake of convenience that
    $\lim_{x \to - \infty} \abs{\delta(x)} = 1 =  \sup_{x \in \R} \abs
    {\delta (x)}$. Define $\varepsilon_1 = 2 \delta$, $\varepsilon_2 =
    (1/2) \delta$, which implies in particular that
\begin{equation}
    \label{eq:22}
    \lim_{x \to -\infty} \abs{\varepsilon_1(x)} > 1 >  \sup_{x \in \R} \abs {\varepsilon_2 (x)}.
\end{equation}

With these data consider the system
    \begin{equation}
        \label{eq:21}
        \begin{aligned}
            \dot x_1 &= - \varepsilon_1(x_1) - \delta(x_2), \\
            \dot x_2 &= \phantom{-} \delta(x_1) - \varepsilon_2(x_2).
        \end{aligned}
    \end{equation}
It is easy to check that the system \eqref{eq:21} is globally
asymptotically stable by considering the Lyapunov function $V_1$ given by
$V_1(x_1,x_2) = \int_0^{x_1} \delta(r) dr + \int_0^{x_2}  \delta(r) dr$. In this case we see that
\begin{equation*}
    \dot {V}_1(x_1,x_2) = - \varepsilon_1(x_1) \delta(x_1) - \varepsilon_2(x_2) \delta(x_2),
\end{equation*}
which is negative for $x = (x_1,x_2) \neq 0$. Now choose a smooth $\rho
\in\K$ with positive derivative and so that $\lim_{r \to \infty}
\rho(r) = 1$. The function $V_2 := \rho \circ V_1$ is then also a
Lyapunov function with compact sublevel sets on its image and such that $\dot V_2(x) = \rho'(V_1(x))\dot V_1(x) < 0$
for $x \neq0$.

We will now extend the Lyapunov function so that the limit for $x_1 \to
-\infty$ has a certain increase property. To this end consider a decreasing
smooth
function $\eta: \R \to [0,1]$ with support in $(-\infty, -c ]$ for
$c>0$. We will choose $c$ sufficiently large later on. We also assume that
$\lim_{x \to -\infty} \eta(x) = 1$.  Now consider the
function
\begin{equation*}
    W(x_1,x_2) = \eta(x_1) \left(1 + \arctan(x_2) \right).
\end{equation*}
Along the solution of \eqref{eq:21} we have
\begin{equation*}
    \dot  W(x_1,x_2) = - \eta'(x_1) \left(1 + \arctan(x_2)\right) (
    \varepsilon_1(x_1) + \delta(x_2)) + \frac{\eta(x_1)}{1+x_2^2} ( \delta(x_1) - \varepsilon_2(x_2))
\end{equation*}
We note that the factors $- \eta'(x_1) \left(1 + \arctan(x_2)\right)$ and
$\frac{\eta(x_1)}{1+x_2^2}$ are both nonnegative. Also $\dot
W(x_1,x_2) = 0$ for $x_1 > -c$. Using \eqref{eq:22} we now choose $c>0$ such that
\begin{equation}
    \label{eq:28}
    \varepsilon_1(-c) < -1 \quad \text {and} \quad \delta(-c) < -  \sup_{x \in \R} \abs {\varepsilon_2 (x)}.  
\end{equation}
With this choice and our assumptions on $\varepsilon_1,\varepsilon_2,\delta$ we have $\dot  W(x_1,x_2) \leq 0$ for all $x \in \R^2$.
For the Lyapunov function $V_3(x) = V_2(x) + W(x)$ we have thus that $\dot
V_3(x) < 0$ for all $x \neq 0$, but the sublevel sets of $V_3:\R^2\to[0,2+\tfrac{\pi}{2})$ are no longer
compact on its image. 

Still Theorem~\ref{t:noncoeLT} implies global asymptotic
stability appealing to $V_3$,
as trajectories exist for all positive times. The latter fact is
obvious from the boundedness of the right hand side of \eqref{eq:21}.

We will now use the previous system as the basis for our
counterexample. To this end let $\psi :\R \to (-1,\infty)$ be an
increasing diffeomorphism with $\psi(x) = x$ for $x\geq 0$ and such that
$\mathrm {ess.}\limsup_{x\to - \infty} \delta'(x)/\psi'(x) < M$ for some constant $M>0$. Then the
transformation 
\begin{equation}
    \label{eq:23}
    T: \R^2 \to (-1,\infty) \times \R , \quad (x,y) \mapsto (\psi(x),y)
\end{equation}
is a diffeomorphism. On $(-1,\infty) \times \R$ we may thus consider the
differential equation
\begin{equation}
    \label{eq:24}
    \dot z = F(z) := DT(T^{-1}(z)) f(T^{-1} (z)) = \left \{
      \begin{matrix}
          \psi'(\psi^{-1}(z_1))(- \varepsilon_1(\psi^{-1}(z_1)) - \delta(z_2)) \\
          \delta(\psi^{-1}(z_1)) - \varepsilon_2(z_2)
      \end{matrix} \right..
\end{equation}
As $\psi'(x_1) \to 0$ for $x_1 \to - \infty$, it follows that $F$ can be
extended continuously to $\R^2$ by setting $F(z_1,z_2) :=
\begin{bmatrix}
    0, & -1 - \varepsilon_2(z_2)
\end{bmatrix}^\top$ for $z \in \R^2$ with $z_1 \leq -1$. The condition
$\delta'(x)/\psi'(x) < M$ for all $x$ sufficiently negative guarantees that this
extension is Lipschitz continuous.

Transforming the Lyapunov function $V_3$ as well, we see that the
transformed version may also be
continuously extended by
\begin{equation}
    \label{eq:25}
    V(z) := \left \{ \ 
      \begin{matrix}
          V_3(T^{-1}(z)) & \quad \text {if }   z_1 > -1, \\ 
           2 + \arctan(z_2) & \quad \text {if } z_1 \leq -1.
      \end{matrix} \right. 
\end{equation}

By construction, we have for the system $\dot z = F(z)$, $z \in \R^2$ that
$ \dot V(z) < 0$ for all $z \neq0$. To arrive at an estimate of the
form~\eqref{DissipationIneq_UGAS_With_Disturbances} we define
\begin{equation*}
    h(z) := \max \left \{ 1, \frac{\norm{z}}{\abs{\dot V(z)}} \right \}
    ,\quad \norm {z} \geq 2, 
\end{equation*}
and extend $h$ to a positive, Lipschitz continuous function on $\R^2$ which is
identically equal to $1$ on $B_1(0)$.
The vector field $F_2 := h F$ is simply a time transformation of
$F$. As $h>0$ everywhere, it follows that the stability properties of
$z^*=0$ for the system
\begin{equation}
    \label{eq:27}
    \dot z = F_2(z)
\end{equation}
 are the same as those of $\dot z = F(z)$. If we consider
the decay of $V$ along the trajectories of \eqref{eq:27} we
obtain (the index denotes the system with respect to which the time
derivative is taken)
\begin{equation}
    \label{eq:26}
    \dot  V_{|F_2}(z) = h(z) \dot V_{|F}(z)  \leq - \norm {z} ,\quad \norm{z}
    \geq 2. 
\end{equation}
As $\dot V_{|F_2}(z) < 0$ for all $z\neq 0$ it is easy to find an $\alpha
\in\mathcal{K}$ such that $\dot  V_{|F_2}(z) \leq - \alpha(\norm {z})$ for
$\norm{z}< 2$. Combining this with \eqref{eq:26}, $\alpha$ may indeed be
chosen so that the decay
condition~\eqref{DissipationIneq_UGAS_With_Disturbances} holds.

It can be checked directly that solutions of \eqref{eq:27} with an initial
condition $z$ with $z_1\leq -1$ explode. This may also be concluded using
the decay condition \eqref{eq:26} together with the observation that for
such solutions $V(\varphi(t,z)) \in (1,3)$ for all $t$ in the interval
existence. ~ \hfill{} $\square$
\end{example}

The next example shows that even in the benign case of
infinite-dimensional linear systems with a bounded generator $A$ an
estimate of the form $\dot{V}(x) \leq - \gamma V(x)$ does not exclude
exponential instability if the Lyapunov function $V$ is not coercive. This
shows that in the non-coercive case it is essential to have a decay
estimate in terms of the norm of $x$ as in \eqref{eq:3}.

\begin{example}
    We consider the Hilbert space $X=\ell^2(\N,\C)$ and denote its
    elements by $x = (x_i)_{i\in \N}$, where $x_i \in \R^i$, $i=
    1,2,\ldots$. By $\|\cdot\|_2$ we denote the Euclidean norm on $\R^i$
    and the induced matrix norm.

   We construct a linear operator $A$ on $X$ as follows. Let $N_i \in
   \C^{i \times i}$ be the nilpotent matrix of order $i-1$ given by
   \begin{equation*}
    N_i =
    \begin{bmatrix}
        0 & 1      & 0     &  0\\
        0 & \ddots & \ddots & 0\\ 
        \vdots & &  0& 1\\
        0    &   \dots & \dots   &0 
    \end{bmatrix}\,.
   \end{equation*}
   Define
   \begin{equation*}
       A_i = - I_i + N_i
   \end{equation*}
   and the operator $A$ by
   \begin{equation*}
       A x = \left( A_i x_i \right)_{i\in \N} \,. 
   \end{equation*}
    As $\|A_i\|_2 \leq 2$ for all $i$, it is clear that $A$ is bounded on $X$.
    To consider the spectrum of $A$, note that for $\lambda \neq -1$ we have
   \begin{equation}
   \label{eq:ex:inveq}
       \left( \lambda I_i - A_i\right)^{-1} = \frac{1}{\lambda + 1} I_i +
       \frac{1}{(\lambda + 1)^2} N_i + \frac{1}{(\lambda + 1)^3} N_i^2 +
       \ldots + \frac{1}{(\lambda + 1)^i} N_i^{i-1}  
   \end{equation}
	Thus for $| \lambda +1 | >1$ it follows for all $i\in \N$ that $\|\left( \lambda I_i -
     A_i\right)^{-1} \|_2 \leq \sum_{k=1}^{\infty}| \lambda +1 |^{-k} =
   \frac{| \lambda +1 |-1}{| \lambda +1 |}$ and consequently,
   $\lambda I - A$ is invertible with bounded inverse. 	
	
	On the other hand, if $\lambda \in
   B_1(-1)$, then we see from \eqref{eq:ex:inveq} that for the $i$th
   standard unit vector in $\R^i$ we have
   \begin{equation*}
      \|\left( \lambda I_i - A_i\right)^{-1} e_i \|_2^2 = \sum_{k=1}^i \frac{1}{(\lambda + 1)^k} \,. 
   \end{equation*}
   This expression tends to $\infty$ as $i\to  \infty$. As a consequence
   $\lambda I - A$ cannot have a bounded inverse on $X$. As the
   spectrum of $A$ is closed, it follows that
   \begin{equation*}
       \sigma(A) = \overline{B_1(-1)} \,.
   \end{equation*}

   In particular $0\in \sigma(A)$, whence the linear system
   \begin{equation}
   \label{eq:ex:LFnc1}
       \dot x = A x
   \end{equation}
   is not exponentially stable in $x^* =0$. We will show that it is
   possible to construct a non-coercive function which decays exponentially
   along all nontrivial solutions.

  As $A_i$ is Hurwitz with eigenvalue $-1$, we may for each $i\in \N$
  choose a symmetric positive definite solution $P_i$ of the Lyapunov inequality
  \begin{equation*}
      A_i^* P_i + P_i A_i \prec - P_i \,.
  \end{equation*}
  By rescaling, if necessary, we may assume that $\|P_i \|_2 = 1$ for all
  $i$. Define the function $V: X \to \R_+$ by setting for $x = (x_i)_{i\in \N}$ 
  \begin{equation*}
      V(x) = \sum_{i=1}^\infty x_i^* P_i x_i \,.
  \end{equation*}
  Using the positive definiteness and the norm bound of the $P_i$, it is easy to see that for $x\neq 0$ we have
  \begin{equation*}
      0 < V(x) \leq \|x\|^2.
  \end{equation*}
  If we consider the decay of $V$ along solutions we have that
  \begin{equation}
	\label{eq:Vdot_2nd_example}
      \dot {V}(x) =  \sum_{i=1}^\infty x_i^* \left(A_i^* P_i + P_i
        A_i\right) x_i \leq - \sum_{i=1}^\infty x_i^* P_i x_i =
      -  V(x).
  \end{equation}
  It follows that $V(e^{At}x) \leq  e^{-t}V(x)$ for all $x\in X$, so that $V$ decays
  exponentially along the solutions of \eqref{eq:ex:LFnc1}. 

  We therefore have now two arguments that show that $V$ is not
  coercive. If $V$ were coercive, it would be a proper Lyapunov function
  and then we could conclude exponential stability of
  \eqref{eq:ex:LFnc1}. But we know from the analysis of the spectrum that
  this is not the case. 

  On the other hand, it is well known that the smallest eigenvalue of
  the matrices $P_i$ tends to $0$ as $i\to \infty$. Choosing eigenvectors $v_i$
  of norm $1$ corresponding to the smallest eigenvalues of $P_i$, and
  letting $y_i \in X$ be the vectors with only $0$ entries except for the
  entries of $v_i$ in position $i$, it
  follows that $V(y_i) \to 0$ but $\|y_i\| =1$ for all $i$.

  As a variant of this example consider operators $A_\varepsilon = A +
  \varepsilon I$ for $\varepsilon \in (0,1/2)$. These operators generate
  exponentially unstable semigroups by our considerations of the spectrum
  of $A$. On the other hand, the calculations analogous to 	\eqref{eq:Vdot_2nd_example} show that
  $V(e^{A_\varepsilon t}x) \leq  e^{(2\varepsilon- 1)t}V(x)$, so that $V$
  is exponentially decaying along solutions of these systems. ~ \hfill{} $\square$
\end{example}

\section{Conclusions}
\label{sec:conclusions}

A classical result in infinite-dimensional stability theory, which can be
proved using a generalized version of Datko's lemma, states that the origin of an
infinite-dimensional linear undisturbed system on a Banach space is
uniformly globally asymptotically stable if and only if there
exists a non-coercive Lyapunov function for this system.

In this paper we prove that the \textit{existence of a non-coercive Lyapunov
function is equivalent to uniformly globally asymptotically stable of the
equilibrium position for nonlinear infinite-dimensional systems
  with disturbances}, provided the system is robustly forward complete and
the equilibrium is robust.  In the linear case these two properties
are always satisfied, but they are essential for the validity of the
theorem for nonlinear systems.  As we show by means of a counterexample, a
nonlinear system of two ordinary differential equations, possessing a
non-coercive Lyapunov function, may fail to be forward complete.  Also it
is essential, that the decay rate of a non-coercive Lyapunov function
along the trajectory is given in terms of the state: $\dot{V}(x) \leq
-\gamma(\|x\|_X)$.  In a further counterexample we show that linear
undisturbed systems admitting a non-coercive function $V$, satisfying a
decay estimate of the form $\dot{V}(x) \leq -\gamma V(x)$, for some
$\gamma>0$, may have exponentially diverging trajectories, even though $V$
decays to zero exponentially along trajectories.
Also we show that the existence of a non-coercive Lyapunov function for a forward complete system ensures that $0$ is uniformly globally weakly attractive.
Finally, a new construction of a non-coercive global Lyapunov function is presented,
which is based on Sontag's $\KL$-Lemma and on Yoshizawa's method.

An interesting direction for future research is the development of
non-coercive Lyapunov tools for input-to-state stability of
infinite-dimensional systems. In the finite-dimensional case Lyapunov
functions which characterize uniform asymptotic stability properties with respect to
disturbances were a key step in this development.

\section*{Acknowledgements}

This research has been supported by the German Research Foundation (DFG) within the project "Input-to-state stability and stabilization of distributed parameter systems" (grant Wi 1458/13-1).

The authors thank Iasson Karafyllis for useful suggestions and comments,
in particular, for making us aware of Theorem~\ref{LipschitzConverseLyapunovTheorem-1}.

\section*{Literature}

\bibliographystyle{amsplain}

{\bibliography{Mir_LitList_TAMS,TAMS}}

\end{document}